\documentclass[10pt]{article}

\usepackage{xspace}
\usepackage{url}
\usepackage{mathtools}
\usepackage{amssymb}
\usepackage{amsthm}
\usepackage{empheq}
\usepackage{latexsym}
\usepackage{enumitem}
\usepackage{eurosym}
\usepackage{dsfont}
\usepackage{appendix}
\usepackage{color} 
\usepackage[unicode]{hyperref}
\usepackage{frcursive}
\usepackage[utf8]{inputenc}
\usepackage[T1]{fontenc}
\usepackage{geometry}
\usepackage{multirow}
\usepackage{todonotes}
\usepackage{lmodern}
\usepackage{anyfontsize}
\usepackage{stmaryrd}
\usepackage{natbib}
\usepackage{cleveref}
\usepackage[english]{babel}
\usepackage[english=british]{csquotes}
\usepackage{mathtools}
\usepackage{accents}
\usepackage{color}
\usepackage{bm}

\setcitestyle{numbers,open={[},close={]}}

\definecolor{red}{rgb}{0.7,0.15,0.15}
\definecolor{green}{rgb}{0,0.5,0}
\definecolor{blue}{rgb}{0,0,0.7}
\hypersetup{colorlinks, linkcolor={red}, citecolor={green}, urlcolor={blue}}
			
\makeatletter \@addtoreset{equation}{section}

\newtheorem{theorem}{Theorem}[section]
\newtheorem{assumption}[theorem]{Assumption}
\newtheorem{corollary}[theorem]{Corollary}

\newtheorem{lemma}[theorem]{Lemma}
\newtheorem{proposition}[theorem]{Proposition}

\newtheorem{remark}[theorem]{Remark}

\DeclareUnicodeCharacter{014D}{\=o}
\newcommand\cA{\mathcal A}

\newcommand\cC{\mathcal C}

\newcommand\cE{\mathcal E}
\newcommand\cF{\mathcal F}

\newcommand\cH{\mathcal H}

\newcommand\cP{\mathcal P}

\newcommand\cU{\mathcal U}

\newcommand\PP{\mathbb P}

\def \E{\mathbb{E}}
\def \F{\mathbb{F}}

\def \H{\mathbb{H}}

\def \N{\mathbb{N}}

\def \P{\mathbb{P}}
\def \Q{\mathbb{Q}}
\def \R{\mathbb{R}}
\def \S{\mathbb{S}}

\def\Ac{\mathcal{A}}

\def\Cc{\mathcal{C}}

\def\Ec{\mathcal{E}}
\def\Fc{\mathcal{F}}

\def\Uc{\mathcal{U}}

\newcommand{\x}{\mathbf{x}}

\newcommand{\xdim}{m}
\newcommand{\bmdim}{d}

\newcommand{\smalltext}[1]{\text{\fontsize{4}{4}\selectfont$#1$}}
\newcommand{\tinytext}[1]{\text{\fontsize{3}{3}\selectfont$#1$}}

\def\eps{\varepsilon}
\def\d{\mathrm{d}}
\DeclareMathOperator*{\argmax}{arg\,max}

\begin{document}
\title{A policy iteration algorithm for non-Markovian control problems}
\author{Dylan Possama\"i\footnote{ETH Z\"urich, Mathematics department, Switzerland. dylan.possamai@math.ethz.ch. This author gratefully acknowledges the support of the SNF project MINT 205121-21981.} $\qquad \qquad$  Ludovic Tangpi\footnote{Princeton University, ORFE and PACM, USA. ludovic.tangpi@princeton.edu. Research partially funded by the NSF CAREER award DMS-2143861.}}

\date{\today}

\maketitle

\abstract{
In this paper, we propose a new policy iteration algorithm to compute the value function and the optimal controls of continuous-time stochastic control problems. The algorithm relies on successive approximations using linear--quadratic control problems which can all be solved explicitly, and only require to solve recursively linear PDEs in the Markovian case. Though our procedure fails in general to produce a non-decreasing sequence like the standard algorithm, it can be made arbitrarily close to being monotone. More importantly, we recover the standard exponential speed of convergence for both the value and the controls, through purely probabilistic arguments which are significantly simpler than in the classical case. Our proof also accommodates non-Markovian dynamics as well as volatility control, allowing us to obtain the first convergence results in the latter case for a state process in multi-dimensions.
}
\setlength{\parindent}{0pt}



\allowdisplaybreaks

\section{Introduction}\label{sec:intro}

An important problem in practical applications of optimal stochastic control theory is the numerical simulation of the value of the problem, as well as optimal control strategies, provided they exist.
This task is typically very challenging and not amenable to standard methods because optimal stochastic control problems are inherently non-convex, infinite dimensional optimisation problems.
A popular alternative consists in transforming the stochastic control problem into a partial differential equation, the celebrated Hamilton--Jacobi--Bellman (HJB for short) equation, or into a forward--backward stochastic differential equation (FBSDE for short).
Unfortunately, both HJB equations and FBSDEs are typically hard to solve beyond linear and toy models.
In fact, the HJB equation is nonlinear and most numerical algorithm are efficient only in small dimensions, whilst FBSDEs are typically efficiently solvable in short time horizons.

\medskip
A well-known approach allowing to get around these difficulties is the so called \emph{policy iteration algorithm} (PIA for short).
This consists in iteratively constructing optimal control strategies along maximisers of the Hamiltonian of the problem.
This procedure allows to approximate the HJB equation by a sequence of linear PDEs, we refer the reader to the seminal papers by \citeauthor*{bellman1954theory} \cite{bellman1954theory,bellman1957dynamic}, \citeauthor*{howard1960dynamic} \cite{howard1960dynamic}, \citeauthor*{puterman1979convergence} \cite{puterman1979convergence}, or \citeauthor*{puterman1994markov} \cite{puterman1994markov,puterman1981convergence}, as well as the more recent takes by \citeauthor*{bertsekas2015value} \cite{bertsekas2015value}, \citeauthor*{jacka2017policy} \cite{jacka2017policy}, {\citeauthor*{kerimkulov2020exponential} \cite{kerimkulov2020exponential}} and {\citeauthor*{kerimkulov2021modified} \cite{kerimkulov2021modified}} for standard control problems, and extensions to so-called mean-field games by \citeauthor*{anahtarci2020value} \cite{anahtarci2020value}, \citeauthor*{cacace2021policy} \cite{cacace2021policy}, and \citeauthor*{camilli2022rates} \cite{camilli2022rates}.
The interest of the research community in PIA has been rekindled by recent progress in reinforcement learning, which have highlighted the power of entropy regularisation in the approximation of stochastic control problems.
See for instance {\citeauthor*{reisinger2021regularity} \cite{reisinger2021regularity}}, {\citeauthor*{wang2020reinforcement} \cite{wang2020reinforcement}}, \citeauthor*{firoozi2022exploratory} \cite{firoozi2022exploratory}, {\citeauthor*{guo2022entropy} \cite{guo2022entropy}}, \citeauthor*{tang2022exploratory} \cite{tang2022exploratory}, \citeauthor*{wei2023continuous2} \cite{wei2023continuous2,wei2023continuous}, \citeauthor*{tran2024policy} \cite{tran2024policy}, or \citeauthor*{wei2024unified} \cite{wei2024unified}.

\medskip.
In this context, a powerful approach recently developed by {\citeauthor*{huang2024convergence} \cite{huang2024convergence}} and {\citeauthor*{ma2024convergence} \cite{ma2024convergence}} begins by considering the \emph{relaxed} version of the control problem, then penalise the objective function using Shannon's entropy.
The penalised problem admits then a unique optimiser given as a Gibbs measure.
The key idea is then to use the density function to construct an iterative scheme that allows to approximate the value (and sometimes  an optimal control) of the penalised problem by a sequence of linear PDEs.
Using PDE techniques, notably based on gradient estimates for second order parabolic PDEs, {\citeauthor*{huang2024convergence} \cite{huang2024convergence}} derived an exponential rate of convergence of the solutions of the linear approximating PDEs to the value function.
Much easier proofs have been recently given by {\citeauthor*{ma2024convergence} \cite{ma2024convergence}} based on ideas reminiscent of the Bismut--Elworthy--Li formula, see \cite{elworthy1994formulae}.

\medskip

The present paper provides a new take on PIA.
We propose a new policy iteration algorithm that is built on fully probabilistic arguments and that allows to approximate both the value of the problem and optimal control strategies.
The method we discuss is standard for policy iteration algorithms studied for instance by {\citeauthor*{kerimkulov2020exponential} \cite{kerimkulov2020exponential}}, in the sense that we approximate candidate optimisers along extremal points of the Hamiltonian.
However, we additionally consider a vanishing penalisation term that allows to speedup the approximation algorithm by explicitly solving the approximating equations.
This allows to approximate the value function (and the optimal control) by explicitly given functions. The main features of the new methods are thus the following:
\begin{enumerate}
	\item[$(i)$] the method is fully probabilistic, and works for non-Markovian control problem.
	In the Markovian case, it has a natural PDE formulation;
	\item[$(ii)$] the method does not rely on relaxation of the control problem. As such, we approximate the optimal (strict) control strategies, when they exist;
	\item[$(iii)$] we do not approximate a penalised problem, but rather the actual stochastic control problem, hence bypassing an additional approximation step.
\end{enumerate}
We will provide a comprehensive comparison of our method to the existing literature in \Cref{sec:comparison}, so that we would rather discuss the major takeaways of the paper at this point.
In our main result, we recover the exponential convergence rate previously obtained in the PIA literature for the value function of the control problem, and we believe our proofs are simple, avoiding subtle and technical PDE results.
Assuming boundedness of the control process $Z$ of the BSDE associated to the control problem, rather straightforward BSDE estimates allow to derive the convergence rate.
When looking at the weak formulation of the problem, the same conditions allow to approximate optimal control processes.
Under stronger conditions, notably in the Markovian case this rate is also valid for the approximation of optimal controls in the strong formulation of the problem.
We should stress here that standard assumptions as those in {\citeauthor*{huang2024convergence} \cite{huang2024convergence}} and {\citeauthor*{ma2024convergence} \cite{ma2024convergence}} guarantee boundedness of the process $Z$ mentioned above.
Moreover, the algorithm discussed in the present paper and our method of proof lend themselves to the study of problems with controlled volatility.
In this setting, we derive an exponential convergence rate under a smallness condition on the coefficients which is automatically satisfied for instance when the terminal reward function is zero, or under some separation condition on the coefficients.
This is an important contribution in our view, as the case of controlled volatility is particularly difficult, and to the best of our knowledge, only {\citeauthor*{ma2024convergence} \cite{ma2024convergence}} treats this case, focusing on one-dimensional dynamics. In particular, our probabilistic arguments allow to avoid difficulties coming with analysing fully nonlinear HJB equation arising from controlled volatility problems.
However, this approach implicitly benefits from the theory of second order BSDEs as in {\citeauthor*{cvitanic2018dynamic} \cite{cvitanic2018dynamic}}.
We refer the reader to \citeauthor*{soner2012wellposedness} \cite{soner2012wellposedness}, \citeauthor*{possamai2018stochastic} \cite{possamai2018stochastic}, or \citeauthor*{zhang2017backward} \cite[Section 12]{zhang2017backward} for a nice exposition.

\medskip

The structure of the rest of the paper is as follows:
in the next section we state the problem and present the main results.
We begin by the uncontrolled volatility case, then we explain how our algorithm can be viewed as an entropy penalisation method and specialise the result to the Markovian case.
We conclude the section by stating the result for the controlled volatility case.
\Cref{sec:proofs} is dedicated to the proofs of the main results.
For the reader's convenience, we present some well-known BSDE characterisations of optimal control problems in the appendix.

\section{Problem statement and main results}
\label{sec:problem.results}

We fix three positive integers $\xdim$, $\bmdim$ and $k$, as well as a finite time horizon $T>0$.
On a probability space $(\Omega,\Fc,\P)$ carrying an $\R^d$-valued Brownian motions $B$, we denote by $\F\coloneqq (\Fc_t)_{t\in[0,T]}$ the $\P$-completed natural filtration of $B$. 
Let $\cC_m$ denote the space of $\R^m$-valued continuous functions on $[0,T]$ equipped with the uniform norm $\|\x\|_\infty \coloneqq  \sup_{t\in [0,T]}\|x_t\|$.
In this paper, we assume to be given a compact subset $U$ of $\R^k$, that we will call the action space, an $\cF_0$-measurable random variable $\xi$ that will be the initial position of the state process as well as four Borel-measurable functions
\begin{align*}
	L&:[0,T]\times \cC_m \times U \longrightarrow \R,\; b: [0,T]\times \cC_m\times U\longrightarrow \R^d,\; 
	\sigma: [0,T]\times \cC_m\times U\longrightarrow \R^{m\times d},\;
	G:\cC_m\longrightarrow \R.
\end{align*}
We will study a form of policy iteration algorithm for stochastic control problems with running and terminal costs respectively given by $L$ and $G$, and state process with coefficients $b$ and $\sigma$.
The set of admissible controls, denoted $\Uc$ is given by
\begin{equation*}
	\Uc \coloneqq \big\{\nu:[0,T]\times\Omega\longrightarrow U: \F\text{--progressively measurable} \big\}.
\end{equation*}
We are interested in the non-Markovian optimal stochastic control problem
\begin{align}
\label{eq:OG.problem}
	V\coloneqq \sup_{\nu \in \Uc}\E^\P\bigg[G(X^\nu) - \int_0^TL\big(s, X^\nu_{\cdot\wedge s}, \nu_s\big)\mathrm{d}s \bigg],
\end{align}
subject to the dynamics
\begin{equation}
\label{eq:controlled.SDE}
	X^\nu_t =\xi+ \int_0^t\sigma\big(s,X^\nu_{\cdot\wedge s},\nu_s\big)b\big(s,X_{\cdot\wedge s}^\nu, \nu_s\big)\mathrm{d}s + \int_0^t\sigma\big(s,X^\nu_{\cdot\wedge s},\nu_s\big)\mathrm{d}B_s, t\in[0,T],\; \P\text{\rm--a.s.}
\end{equation}
We will give conditions below guaranteeing that this problem is well-posed and that  $V<\infty$.
In our analysis, we will separately consider the cases where the volatility coefficient $\sigma$ is uncontrolled (\emph{i.e.} does not depend on the third argument) and the case where it is.

\subsection{The case of uncontrolled volatility}
\label{sec:no.vol}
In this subsection, let us assume that $\sigma:[0,T]\times \cC_m\longrightarrow \R^{m\times d}$ does not depend on the control argument $u$.
In this case, the---reduced---Hamiltonian $h$ of the problem is given by
\begin{equation*}
	h(t,\x,z,u) \coloneqq b(t, \x, u)\cdot z - L(t, \x,u),\; (t,\x,z,u)\in[0,T]\times\Cc_m\times\R^d\times U,
\end{equation*}
and the optimised Hamiltonian by
\begin{equation*}
	H(t,\x, z) \coloneqq  \sup_{u\in U}\big\{h(t, \x,z,u)\big\}, (t.\x,z)\in[0,T]\times\Cc_m\times\R^d.
\end{equation*}
We will work under the following conditions.
\begin{assumption}\label{assump:standing}
	\begin{enumerate}
	\item[$(i)$] the functions $b$ and $\sigma$ are bounded$;$
	\item[$(ii)$] the functions $L$ and $G$ are of sub-quadratic growth in the sense that there is some $C>0$ and some $\eps\in(0,2]$ such that
	\begin{equation*}
		|G(\x)| + |L(t,\x,u)| \le C\big(1+\|\x\|_\infty^{2-\varepsilon}\big),\; \forall (t,\x,u)\in [0,T]\times \cC_m\times U,
	\end{equation*}
	and the functions $L$ and $G$ are lower-semicontinuous in $(\x,u)$ and $\x$, respectively$;$
	\item[$(iii)$] the functions $\sigma b$ and $\sigma$ are Lipschitz-continuous in $\x$ uniformly in the other arguments$;$
	\item[$(v)$] there is a Borel-measurable map $u^\star:[0,T]\times\Cc_m\times\R^d\longrightarrow U$ satisfying
	\begin{equation}
	\label{eq:argmax.ass}
	 	u^\star(t,\x,z) \in \argmax_{u\in U}\big\{h(t, \x,z,u)\big\} ,\; (t,\x,z)\in[0,T]\times\Cc_m\times\R^d,
	 \end{equation} 
	 which is in addition Lipschitz-continuous in its last argument, uniformly in the others.
\end{enumerate}
\end{assumption}
In the rest of this section we fix $X$ to be the unique strong solution of the driftless SDE 
\begin{equation}\label{eq:driftless}
		X_t= \xi +\int_0^t\sigma(s,X_{\cdot\wedge s})\mathrm{d}B_s,\; t\in[0,T],\; \P\text{\rm--a.s.}
\end{equation}
It is not hard to convince oneself that under the above conditions, the value $V$ of the control problem is finite, since the following estimate holds.
\begin{lemma}\label{lem:estimX}
For any $p\geq 1$, we have
\[
\E^\P\bigg[\exp\bigg(p\sup_{t\in[0,T]}\|X_t\|^{2-\eps}\bigg)\bigg]<+\infty.
\]
\end{lemma}
\begin{proof}
Given that $\sigma$ is bounded by some constant $M>0$, and since the map $x\longmapsto \exp(p\|x\|^{2-\eps})$ is convex for $p>0$, we can use \citeauthor*{pages2014convex} \cite[Proposition 4]{pages2014convex} to deduce that
\[
\E^\P\bigg[\exp\bigg(p\sup_{t\in[0,T]}\|X_t\|^{2-\eps}\bigg)\bigg]\leq \E^\P\bigg[\exp\bigg(pM^{2-\eps}\sup_{t\in[0,T]}\|B_t\|^{2-\eps}\bigg)\bigg]<+\infty,
\]
where the finiteness of the right-hand side above is immediate from properties of the Gaussian distribution.
\end{proof}

We are now ready to define the policy iteration algorithm that we introduce in this work.
It proceeds as follows: we fix a Borel-measurable function $u^\star:[0,T]\times \cC_m\times \R^d\longrightarrow U$ from \Cref{assump:standing}.$(v)$, and a non-decreasing function $\varphi:\R_+ \longrightarrow \R_+^\star$ such that $\varphi(0) = 1$.
We also let $\Ac$ be the set of $\R^d$-valued, $\F$--progressively measurable processes $\alpha$ such that the stochastic exponential $\Ec\big(\int_0^\cdot \alpha_s\cdot\mathrm{d}B_s\big)$ is a uniformly integrable $(\F,\P)$-martingale on $[0,T]$.

\paragraph*{Step $0$.}
Consider the control problem
\begin{equation*}
	\begin{cases}
	\displaystyle V^0 \coloneqq  \sup_{\alpha \in \cA}\E^{\PP^{\smalltext{\alpha}\smalltext{,}\smalltext{0}}}\bigg[G(X) - \int_0^T\bigg(\frac12\|\alpha_s\|^2 + L(s,X_{\cdot\wedge s},0)\bigg)\mathrm{d}s  \bigg],\\[0.8em]
	\displaystyle\frac{\mathrm{d}\PP^{{\alpha,0}}}{\mathrm{d}\PP}\coloneqq \cE\bigg(\int_0^\cdot \big( b(s,X_{\cdot\wedge s},0) + \alpha_s\big)\cdot \mathrm{d} B_s\bigg)_T.
	\end{cases}
\end{equation*}
By \Cref{prop:Bellman2}, we have $V^0 = Y^0_0$ where $( Y^0,Z^0)$ solves the BSDE
\begin{equation}
\label{eq:BSDE0}
	Y^0_t = F^0(X) + \int_t^T\bigg(\frac12\|Z_s^0\|^2 + b(s,X_{\cdot\wedge s},0)\cdot Z_s^0\bigg) \mathrm{d}s - \int_t^TZ_s^0\cdot \mathrm{d}B_s,\; t\in[0,T],\; \P\text{\rm--a.s.},
\end{equation}
with
\begin{equation*}
	F^0(X) \coloneqq  G(X) - \int_0^TL(s,X_{\cdot\wedge s},0)\mathrm{d}s.
\end{equation*}
The value process $Y^0$ of this equation can be computed explicitly.
Indeed, as we will see below, a change of measures and an application of It\^o's formula yield
\begin{equation*}
	Y^0_t = \log\Big(\E^{\PP^\smalltext{0}}\big[\exp\big(F^0(X)\big) \big|\cF_t \big] \Big),\; t\in[0,T],\; \text{with}\; \frac{\mathrm{d}\PP^0}{\mathrm{d}\PP} \coloneqq  \cE\bigg(\int_0^\cdot b(s,X_{\cdot\wedge s},0)\cdot \mathrm{d}B_s\bigg)_T.
\end{equation*}

\paragraph*{Step $n\in\N^\star$.}
Given that the construction from step $0$ has been iterated until the step $n-1$ for some $n\in\N^\star$, giving us in particular a collection of pair of processes $(Y^k,Z^k)_{k\in\{0,\dots,n-1\}}$, we consider the control problem
\begin{equation}
\label{eq:problem.n}
	\begin{cases}
		\displaystyle V^n \coloneqq  \sup_{\alpha\in \cA}\E^{\PP^{\smalltext{\alpha}\smalltext{,}\smalltext{n}}}\bigg[G(X) - \int_0^T\bigg(\frac12\|\alpha_s\|^2 + L\big(s,X_{\cdot\wedge s}, u^\star(s, X_{\cdot\wedge s}, Z^{n-1}_s)\big)\bigg)\mathrm{d}s \bigg],\\[0.8em]
	\displaystyle	\frac{\mathrm{d}\PP^{\alpha,n}}{\mathrm{d}\PP}\coloneqq \cE\bigg(\int_0^\cdot \bigg( b\big(s,X_{\cdot\wedge s},u^\star(s, X_{\cdot\wedge s}, Z^{n-1}_s)\big) + \frac{\alpha_s}{\varphi^{1/2}(n)}\bigg)\cdot \mathrm{d} B_s\bigg)_T.
	\end{cases}
\end{equation}
Then, by \Cref{prop:Bellman2} the following BSDE admits a unique solution (in appropriate spaces)
\begin{equation}
\label{eq:BSDE.n}
	Y^n_t = F^n(X) + \int_t^T\bigg(\frac{1}{2\varphi(n)}\|Z_s^n\|^2+ b\big(s,X_{\cdot\wedge s},u^\star(s, X_{\cdot\wedge s}, Z^{n-1}_s)\big)\cdot Z^n_s\bigg)\mathrm{d}s - \int_t^TZ_s^n\cdot \mathrm{d}B_s, \; t\in [0,T], \; \PP\text{\rm--a.s.},
\end{equation}
where
\begin{equation*}
	F^n(X) \coloneqq G(X) - \int_0^TL\big(s,X_{\cdot\wedge s},u^\star(s, X_{\cdot\wedge s}, Z^{n-1}_s)\big)\mathrm{d}s,
\end{equation*}
and the problem \eqref{eq:problem.n} is such that $V^n= Y^n_0$. The value process $Y^n$ is again explicitly computed as
\begin{equation}
\label{eq:def.Pn}
	 Y^n_t = \varphi(n)\log\bigg(\E^{\PP^\smalltext{n}}\bigg[\exp\bigg(\frac{F^n(X)}{\varphi(n)}\bigg)\bigg|\cF_t\bigg] \bigg),\; t\in[0,T],\; \text{with}\; \frac{\mathrm{d}\PP^{n}}{\mathrm{d}\PP}\coloneqq \cE\bigg(\int_0^\cdot  b\big(s,X_{\cdot\wedge s},u^\star(s, X_{\cdot\wedge s}, Z^{n-1}_s)\big)\cdot \mathrm{d} B_s\bigg)_T.
\end{equation}
Observe that the iteration scheme $(V^n, Y^n,Z^n,\alpha^n)_{n\in\N}$ depends on the choice of $u^\star$ and $\varphi$ and that for each such $u^\star$ and $\varphi$, it is well-defined.

\medskip

Before going any further, let us stress that the computation of $Y^n$ is almost explicit, as it reduces to computation of conditional expectations, provided that one knows $Z^{n-1}$ of course. This can be achieved extremely efficiently numerically, using ones favourite method for solving BSDEs, for which the literature is vast, parsing Monte-Carlo techniques, branching processes, or machine learning (see among many others \citeauthor*{bouchard2004discrete} \cite{bouchard2004discrete}, \citeauthor*{zhang2004numerical} \cite{zhang2004numerical}, \citeauthor*{gobet2005regression} \cite{gobet2005regression}, \citeauthor*{gobet2016approximation} \cite{gobet2016approximation,gobet2017adaptive}, \citeauthor*{gobet2016stratified} \cite{gobet2016stratified}, \citeauthor*{gobet2020quasi} \cite{gobet2020quasi}, \citeauthor*{crisan2014second} \cite{crisan2010solving,crisan2012solving,crisan2014second}, \citeauthor*{chassagneux2014linear} \cite{chassagneux2014linear}, \citeauthor*{chassagneux2014runge} \cite{chassagneux2014runge}, \citeauthor*{possamai2015weak} \cite{possamai2015weak},  \citeauthor*{bouchard2016numerical}, \cite{bouchard2016numerical}, \citeauthor*{bouchard2019numerical} \cite{bouchard2019numerical}, \citeauthor*{e2017deep} \cite{e2017deep}, \citeauthor*{hutzenthaler2020proof} \cite{hutzenthaler2020proof}, \citeauthor*{hure2020deep} \cite{hure2020deep}, \citeauthor*{germain2020deep} \cite{germain2020deep}, and \citeauthor*{beck2023overview} \cite{beck2023overview}. 

\medskip
We emphasise that is is also possible to compute $Z^{n-1}$ in terms of $Y^{n-1}$ using techniques from Malliavin calculus, see \emph{e.g.} \citeauthor*{nualart2006malliavin} \cite{nualart2006malliavin} for more details. 
Provided that $G(X)$, $b(t,X_{\cdot\wedge t},u)$, $L(t,X_{\cdot\wedge t},u)$ and $u^\star(t,X_{\cdot\wedge t},z)$ are all Malliavin differentiable, for any $(t,u,z)\in[0,T]\times U\times\R^d$, and are uniformly Lipschitz-continuous with respect to the variables $u$ and $z$, it is known that the solution to BSDE \eqref{eq:BSDE0} 
will be Malliavin differentiable (see for instance the ideas in \citeauthor*{mastrolia2014malliavin} \cite{mastrolia2014malliavin}), and that $Z^0_t = D_tY^0_t$, from which we will deduce by induction that for all $n$, $(Y^n,Z^n)$ is Malliavin differentiable and
$Z^n_t=D_t Y_t^n,\; \mathrm{d}t\otimes\P\text{\rm--a.e.}$, see the proof of \Cref{prop:pointwise.conv} for more details on this.

\medskip
It is then enough to differentiate \Cref{eq:def.Pn}, bringing first the expectation back to the measure $\P$, to write for any $(r,i)\in[0,T]\times\{1,\dots,d\}$, using in particular the chain rule, see \cite[Proposition 1.2.3]{nualart2006malliavin}, as well as \cite[Proposition 1.2.8]{nualart2006malliavin}
\begin{align*}
 D^i_r Y^n_t &= \varphi(n)D^i_r\log\bigg(\E^{\PP}\bigg[\mathrm{e}^{\int_\smalltext{t}^\smalltext{T}b(s,X_{\smalltext{\cdot}\smalltext{\wedge} \smalltext{s}},u^\smalltext{\star}(s, X_{\smalltext{\cdot}\smalltext{\wedge} \smalltext{s}}, Z^{\smalltext{n}\smalltext{-}\smalltext{1}}_\smalltext{s}))\cdot\mathrm{d}B_\smalltext{s}-\int_\smalltext{t}^\smalltext{T}\frac{\|b(s,X_{\smalltext{\cdot}\smalltext{\wedge} \smalltext{s}},u^\smalltext{\star}(s, X_{\smalltext{\cdot}\smalltext{\wedge} \smalltext{s}}, Z^{\smalltext{n}\smalltext{-}\smalltext{1}}_\smalltext{s}))\|^\smalltext{2}}{2}\mathrm{d}s+\frac{G(X_{\smalltext{\cdot}\smalltext{\wedge} \smalltext{T}})}{\varphi(n)}-\int_0^T\frac{L(s,X_{\smalltext{\cdot}\smalltext{\wedge} \smalltext{s}},u^\smalltext{\star}(s, X_{\smalltext{\cdot}\smalltext{\wedge}\smalltext{ s}}, Z^{\smalltext{n}\smalltext{-}\smalltext{1}}_\smalltext{s}))}{\varphi(n)}\mathrm{d}s}\bigg|\cF_t\bigg] \bigg)\\
 &=\varphi(n)\frac{\E^{\PP}\Big[D_r^i\mathrm{e}^{\int_\smalltext{t}^\smalltext{T}b(s,X_{\smalltext{\cdot}\smalltext{\wedge} \smalltext{s}},u^\smalltext{\star}(s, X_{\smalltext{\cdot}\smalltext{\wedge} \smalltext{s}}, Z^{\smalltext{n}\smalltext{-}\smalltext{1}}_\smalltext{s}))\cdot\mathrm{d}B_\smalltext{s}-\int_\smalltext{t}^\smalltext{T}\frac{\|b(s,X_{\smalltext{\cdot}\smalltext{\wedge} \smalltext{s}},u^\smalltext{\star}(s, X_{\smalltext{\cdot}\smalltext{\wedge} \smalltext{s}}, Z^{\smalltext{n}\smalltext{-}\smalltext{1}}_\smalltext{s}))\|^\smalltext{2}}{2}\mathrm{d}s+\frac{G(X_{\smalltext{\cdot}\smalltext{\wedge} \smalltext{T}})}{\varphi(n)}-\int_0^T\frac{L(s,X_{\smalltext{\cdot}\smalltext{\wedge} \smalltext{s}},u^\smalltext{\star}(s, X_{\smalltext{\cdot}\smalltext{\wedge}\smalltext{ s}}, Z^{\smalltext{n}\smalltext{-}\smalltext{1}}_\smalltext{s}))}{\varphi(n)}\mathrm{d}s}\Big|\cF_t\Big]}{\E^{\PP}\Big[\mathrm{e}^{\int_\smalltext{t}^\smalltext{T}b(s,X_{\smalltext{\cdot}\smalltext{\wedge} \smalltext{s}},u^\smalltext{\star}(s, X_{\smalltext{\cdot}\smalltext{\wedge} \smalltext{s}}, Z^{\smalltext{n}\smalltext{-}\smalltext{1}}_\smalltext{s}))\cdot\mathrm{d}B_\smalltext{s}-\int_\smalltext{t}^\smalltext{T}\frac{\|b(s,X_{\smalltext{\cdot}\smalltext{\wedge} \smalltext{s}},u^\smalltext{\star}(s, X_{\smalltext{\cdot}\smalltext{\wedge} \smalltext{s}}, Z^{\smalltext{n}\smalltext{-}\smalltext{1}}_\smalltext{s}))\|^\smalltext{2}}{2}\mathrm{d}s+\frac{G(X_{\smalltext{\cdot}\smalltext{\wedge} \smalltext{T}})}{\varphi(n)}-\int_0^T\frac{L(s,X_{\smalltext{\cdot}\smalltext{\wedge} \smalltext{s}},u^\smalltext{\star}(s, X_{\smalltext{\cdot}\smalltext{\wedge}\smalltext{ s}}, Z^{\smalltext{n}\smalltext{-}\smalltext{1}}_\smalltext{s}))}{\varphi(n)}\mathrm{d}s}\Big|\cF_t\Big] }.
\end{align*}
Then, denoting for simplicity
\[
A^n_t\coloneqq \mathrm{e}^{\int_\smalltext{t}^\smalltext{T}b(s,X_{\smalltext{\cdot}\smalltext{\wedge} \smalltext{s}},u^\smalltext{\star}(s, X_{\smalltext{\cdot}\smalltext{\wedge} \smalltext{s}}, Z^{\smalltext{n}\smalltext{-}\smalltext{1}}_\smalltext{s}))\cdot\mathrm{d}B_\smalltext{s}-\int_\smalltext{t}^\smalltext{T}\frac{\|b(s,X_{\smalltext{\cdot}\smalltext{\wedge} \smalltext{s}},u^\smalltext{\star}(s, X_{\smalltext{\cdot}\smalltext{\wedge} \smalltext{s}}, Z^{\smalltext{n}\smalltext{-}\smalltext{1}}_\smalltext{s}))\|^\smalltext{2}}{2}\mathrm{d}s+\frac{G(X_{\smalltext{\cdot}\smalltext{\wedge} \smalltext{T}})}{\varphi(n)}-\int_0^T\frac{L(s,X_{\smalltext{\cdot}\smalltext{\wedge} \smalltext{s}},u^\smalltext{\star}(s, X_{\smalltext{\cdot}\smalltext{\wedge}\smalltext{ s}}, Z^{\smalltext{n}\smalltext{-}\smalltext{1}}_\smalltext{s}))}{\varphi(n)}\mathrm{d}s},
\]
we have again using the chain rule (where we denote by $f_x$ the derivative of any function $f$ with respect of its $x$-variable )
\begin{align*}
\frac{D_r^iA^n_t}{A^n_t}&=\int_{t}^{T}\Big(D_r^ib\big(s,X_{{\cdot}{\wedge} {s}},u^{\star}(s, X_{{\cdot}{\wedge} {s}}, Z^{{n}{-}{1}}_{s})\big)+b_u\big(s,X_{{\cdot}{\wedge} {s}},u^{\star}(s, X_{{\cdot}{\wedge} {s}}, Z^{{n}{-}{1}}_{s})\big)\big((D_r^iu^{\star}+u^{\star}_z))(s, X_{{\cdot}{\wedge} {s}}, Z^{{n}{-}{1}}_{s}) D_r^iZ^{n-1}_s \Big)\cdot\mathrm{d}B_{s}\\
&\quad-\int_{t}^{T}b\big(s,X_{{\cdot}{\wedge} {s}},u^{\star}(s, X_{{\cdot}{\wedge} {s}}, Z^{{n}{-}{1}}_{s})\big)\cdot D_r^ib(s,X_{{\cdot}{\wedge} {s}},u^{\star}(s, X_{{\cdot}{\wedge} {s}}, Z^{{n}{-}{1}}_{s}))\mathrm{d}s\\
&\quad+\int_t^Tb\big(s,X_{{\cdot}{\wedge} {s}},u^{\star}(s, X_{{\cdot}{\wedge} {s}}, Z^{{n}{-}{1}}_{s})\big)\cdot b_u\big(s,X_{{\cdot}{\wedge} {s}},u^{\star}(s, X_{{\cdot}{\wedge} {s}}, Z^{{n}{-}{1}}_{s})\big)\big((D_r^iu^{\star}+u^{\star}_z))(s, X_{{\cdot}{\wedge} {s}}, Z^{{n}{-}{1}}_{s})D_r^iZ^{n-1}_s\mathrm{d}s+\frac{D_r^iG(X)}{\varphi(n)}\\
&\quad-\int_{t}^{T}\Big(D_r^iL\big(s,X_{{\cdot}{\wedge} {s}},u^{\star}(s, X_{{\cdot}{\wedge} {s}}, Z^{{n}{-}{1}}_{s})\big)+L_u\big(s,X_{{\cdot}{\wedge} {s}},u^{\star}(s, X_{{\cdot}{\wedge} {s}}, Z^{{n}{-}{1}}_{s})\big)\big((D_r^iu^{\star}+u^{\star}_z))(s, X_{{\cdot}{\wedge} {s}}, Z^{{n}{-}{1}}_{s}) D_r^iZ^{n-1}_s \Big)\mathrm{d}s.
\end{align*}
This allows to obtain a recursive scheme to compute $Z^n$ provided that we know $Z^{n-1}$ and $D_rZ^{n-1}$. Since the above formula for $n=0$ simplifies greatly (because $Z^{-1}=0$), the initialisation can effectively be done. 

\medskip

Next, to state the main result of this work, we will need to introduce an additional BSDE, which will naturally appear as the limit of the previously defined recursive scheme
\begin{equation}\label{bsde}
	Y_t=G(X)+\int_t^TH(s,X_{\cdot\wedge s},Z_s)\mathrm{d}s-\int_t^TZ_s\cdot\mathrm{d}B_s,\; t\in[0,T),\; \P\text{\rm--a.s.}
\end{equation}
We will show later that under our assumptions, the above equation is well-defined, and that its solution satisfies natural integrability requirements, see \Cref{prop:Bellman} and \Cref{lemma:unif}. We however require more.
\begin{assumption}
\label{ass.z.bounded}
	The unique solution to {\rm BSDE} \eqref{bsde} is such that $Z$ is bounded.
\end{assumption}

\begin{remark}
\label{rem:bounded.z}
{\rm \Cref{ass.z.bounded}} is admittedly an implicit assumption on the data of the control problem. In the Markovian case, it would be equivalent to asking that the value function of the control problem has a bounded first-order derivative, which is actually weaker than what is assumed in {\rm \citeauthor*{huang2024convergence} \cite{huang2024convergence}} or {\rm \citeauthor*{ma2024convergence} \cite{ma2024convergence}}. Notwithstanding, in the non-Markovian case there are sufficient conditions guaranteeing boundedness of $Z$, for instance by {\rm\citeauthor*{cheridito2014bsdes} \cite{cheridito2014bsdes}} this would hold if
\begin{enumerate}
		\item[$(i)$] $G(X)$ is a Malliavin differentiable random variable and there is a constant $C_G$ such that for any $i\in\{1,\dots,d\}$ and any $t\in[0,T]$, $|D^i_tG(X)|\le C_G$, $\PP\otimes \mathrm{d}t$--{\rm a.e.}$;$
		\item[$(ii)$] for any $i\in\{1,\dots,d\}$, there is a $\P$--square-integrable, Borel-measurable function $q_i:[0,T]\longrightarrow \R_+$ such that for all $z\in\R^d$ with $\|z\|^2 \le Q\coloneqq d\big(C_G +\int_0^Tq_i(t)\mathrm{d}t\big)^2$, the process $H(t, X_{\cdot\wedge t}, z)$ is Malliavin differentiable and for any $r\in[0,T]$, $D^i_rH(t,X_{\cdot\wedge t},z)\le q_i(t)$, $\PP\otimes \mathrm{d}t${\rm--a.e.;}
		\item[$(iii)$] for Lebesgue--almost every $r\in[0,T]$  there exists a non-negative process $K_r\in \H^4(\R,\F,\P)$ such that
		\begin{equation*}
			\int_0^T\|K_{r\cdot}\|^4_{\H^\smalltext{4}(\R,\F,\P)}\mathrm{d}r<+\infty,\; \big\|D_rH(t,X,z) - D_rH(t,X,z^\prime) \big\| \le K_{rt}\|z-z^\prime\|,\; \text{\rm $\forall(z,z^\prime)\in\R^d\times\R^d$ with $\max\{\|z\|,\|z^\prime\|\}\le \sqrt{Q}.$}
		\end{equation*}
	\end{enumerate}
\end{remark}
\begin{theorem}
\label{thm:Rate}
	Let {\rm\Cref{assump:standing}} and {\rm\Cref{ass.z.bounded}} be satisfied.
	Then for each function $u^\star$ 
	satisfying \eqref{eq:argmax.ass}, defining the scheme $(V^n,\widehat Y^n,Z^n,\alpha^n)_{n\in\N}$ through $u^\star$ and any $\varphi$, we have for some $C>0$
	\begin{equation}
	\label{eq:convergence.statement}
		V^n =  \widehat Y^n_0,\; -\frac{C}{2^{n}}\le V^n- V \le \frac{C}{\varphi(n)},\; 
		\forall n\in\N.
	\end{equation}
\end{theorem}
The proof of this result as well as those of all the results stated in the section are presented in \Cref{sec:proofs}.
Some comments about this result are in order. \Cref{assump:standing} consists essentially of growth conditions allowing to guarantee that the problem \eqref{eq:problem.n} is well-posed.
%
Provided that \Cref{ass.z.bounded} is satisfied, no regularity conditions is needed to derive \Cref{thm:Rate}.
As pointed out in \Cref{rem:bounded.z}, some Malliavin differentiability conditions would allow to verify \Cref{ass.z.bounded}, and, in the Markovian case, these conditions would translate in Lipschitz-continuity properties for $G$ and $H$, but not without necessarily needing smoothness.

\medskip

\Cref{thm:Rate} allows to estimate the general control problem \eqref{eq:OG.problem} by the sequence of iterated linear--quadratic control problems \eqref{eq:problem.n}.
These linear--quadratic control problems are essentially explicitly solved, and give a convergence rate of $1/2^{n}$ to the value of the problem of interest.
The approximation of an optimal control using the same scheme will also be given below first with respect to a norm on the product space $[0,T]\times\Omega$, see \Cref{prop:Norm.conv}.
Under stronger condition, this result will be refined to a pointwise rate for each $t\in [0,T]$, see \Cref{prop:pointwise.conv}.
In \Cref{cor:HJB} and Subsection \ref{sec.pointwise.conv.optim.cont} we discuss approximation of optimal control and in \Cref{thm:Rate.vol} we extend the above result to the controlled volatility case.
Before doing that, let us explain how our results and method compares to the existing literature.

\subsection{Comparison to the existing literature}
\label{sec:comparison}
As mentioned in the introduction, the papers developing policy iteration algorithms for stochastic control problems have approached the issues from a rather different perspective.
See for instance {\rm \citeauthor*{huang2024convergence} \cite{huang2024convergence}}, \citeauthor*{ma2024convergence} \cite{ma2024convergence} and references therein.
The method favoured by these papers begins with idea popularised in the reinforcement learning literature.
One rather optimises the desired reward function over the set of \emph{relaxed controls}, and then solves the regularised--penalised problem by Shannon's entropy in order to capture the trade-off between exploitation and exploration.
From the practical point of view, the entropy-regularised problem admits a unique optimal control that takes the form of a Gibb's measure.
This is the basis of the iteration algorithm via a sequence of linear PDEs that can be seen as a Picard iteration of sorts for the HJB equation of the regularised control problem.
As a byproduct of this approach, the optimal control attained in the limit is of \emph{relaxed} form.
Let us now briefly compare our approach and results to the method just described. To ease the comparison, we focus on the Markovian case considered in the existing literature.
That is, we assume that the functions involved
\begin{align*}
	L&:[0,T]\times \R^m \times U \longrightarrow \R,\; b: [0,T]\times \R^m\times U\longrightarrow \R^d,\; \sigma: [0,T]\times \R^m\longrightarrow \R^{m\times d},\;
	G:\R^m\longrightarrow \R,
\end{align*}
depend on the current value of the state process and that $\xi = x_0\in \R^d$.

\paragraph*{Reformulation of the scheme as entropic penalisation.}
By standard results from BSDE theory, see \emph{e.g.} \cite{ma1999forward}, the solution $Z^0$ of the BSDE \eqref{eq:BSDE0} is such that there is a Borel-measurable function $v^0:[0,T]\times \R^m\longrightarrow \R^d$ such that $Z^0_t = v^0(t,X_{t})$.
Therefore, it follows by induction that for all $n\in \N$, there is $v^n:[0,T]\times \R^m\longrightarrow\R^d$ such that $Z^n_t = v^n(t,X_{t})$.
On the other hand, recall the relative-entropy (Kullback--Leibner divergence) $\cH(\Q|\PP)$ of $\Q$ with respect to $\PP$ given by
\begin{equation*}
	\cH(\Q|\PP) \coloneqq \begin{cases}
		\displaystyle \E^{\Q}\bigg[\log\bigg(\frac{\d\Q}{\d\PP}\bigg)\bigg],\; \text{if}\; \Q\ll \PP,\\
		+\infty,\; \text{otherwise}.
	\end{cases}
\end{equation*}
These allow to construct the following iteration scheme: we let $v^{-1} \coloneqq 0$ and assume that for each $n$, the following SDE admits a unique strong solution
\begin{equation}
\label{eq:Non.Markov.SDE}
	X^n_t =\xi+ \int_0^t\sigma(s,X^n_{s})b\big(s,X^n_{ s},u^\star(s, X^n_{ s}, v^{n-1}(s,X^n_{ s}))\big)\d s + \int_0^t\sigma(s,X^n_{s})\d B_s,\; t\in[0,T],\; \P\text{\rm--a.s.}
\end{equation}
Then, for each $n\in \N$, we put
\begin{equation}
\label{eq:problem.n.tilde}
	\begin{cases}
		\displaystyle\widetilde V^n \coloneqq  \sup_{\alpha \in \cA}\bigg\{ \E^{\PP^{\smalltext\nu}}\bigg[G(X^n_{ T}) - \int_0^T L\big(s,X^n_{s}, u^\star(s, X^n_{s}, v^{n-1}(s,X^n_{s}))\big)\d s \bigg] - \varphi^{1/2}(n)\cH(\PP^\alpha|\PP)\bigg\}, \\[0.8em]
		\displaystyle\frac{\mathrm{d}\PP^{\nu}}{\mathrm{d}\PP}\coloneqq \cE\bigg(\int_0^\cdot \alpha_s \cdot \d B_s\bigg)_T.
	\end{cases}
\end{equation}

The following holds.
\begin{proposition}
\label{prop:entropic.penal}
	For every $n\in\N$, we have $\widetilde V^n = V^n$.
\end{proposition}
The assumption on existence of strong solutions of the SDE \eqref{eq:Non.Markov.SDE} is to ease the presentation, as the goal here is simply to showcase an interpretation of our scheme as an entropy-regularised problem.
In a more general presentation, we would work with weak solutions of this SDE, for instance as done in \Cref{sec:controlled.diffusion}.

\medskip

Let us now highlight the following three important features of our approach when compared to the usual method studied in \citeauthor*{ma2024convergence} \cite{ma2024convergence}
\begin{enumerate}
\item[$(i)$] \Cref{thm:Rate} and \Cref{prop:entropic.penal} show that $V$ is actually approximated by a sequence of entropy-penalised problems similar to the existing literature.
However, we do not consider relaxed controls and we penalise the control problems using the relative entropy rather than the Shannon entropy.
Therefore, the optimal control that we approximate is a \emph{strict control process} (it is a weak optimal control in the general non-Markovian case of \Cref{prop:Norm.conv}, and a strong optimal control in the Markovian case of \Cref{cor:HJB}).

\item[$(ii)$] The existing literature uses the PIA to approximate the penalised problem.
Nothing is said about the (optimal) choice of the exogenous penalisation parameter, or how fast the penalised problem converges to the original problem under study.
This is usually a nontrivial problem of intrinsic interest, see for instance {\rm \citeauthor*{reisinger2021regularity} \cite[Section 6]{reisinger2021regularity}} for an in-depth discussion.
Our approach naturally blends penalisation and iteration at the same time, so that we do not need to introduce an additional tuning parameter.

\item[$(iii)$] The PIA used in \citeauthor*{ma2024convergence} \cite{ma2024convergence} is naturally monotone, so that the value function is approximated from below by a sequence of solutions of linear PDEs. 
In contrast, our approximating values do not seem to be monotone. 
However, as \Cref{thm:Rate} shows, it holds $V^n \le V +C/\varphi(n)$ where $\varphi(n)$ can be made to converge arbitrarily fast to infinity. 
This means that the non-monotone part of the approximation can be made arbitrarily small, arbitrarily fast. 
\end{enumerate}

\paragraph*{Reformulation of the scheme as iterated linear HJB equations.}
Further observe that, although the proof we give of \Cref{thm:Rate} is purely probabilistic, a natural PDE formulation can be given in the spirit of \citeauthor*{ma2024convergence} \cite{ma2024convergence}.

\medskip
By (the proof of) \Cref{thm:Rate}, for any $n\in\N$, the value $V^n$ of the problem \eqref{eq:problem.n} satisfies $V^n =  Y^n_0$ and its optimal control $\alpha^n_t = Z^n_t/\varphi^{1/2}(n)$ where $(Y^n,Z^n)$ solves the BSDE \eqref{eq:BSDE.n}.
But in the current Markovian case, it is well-known, that if $b$, $G$, $L$ and $\sigma$ are continuously differentiable, then by {\rm \citeauthor*{ma2002representation} \cite[Theorem 3.1]{ma2002representation}} it holds 
\begin{equation*}
	\widehat Y^n_t = v^n(t,X_t), \; \text{and}\; Z^n_t = (\sigma^\top\nabla v^n)(t,X_t),
\end{equation*}
where $v^n:[0,T]\times \R^m\longrightarrow \R$ is a (differentiable) viscosity solution of the PDE
\begin{equation*}
	\begin{cases}
	\displaystyle v^n_t(t,x)+ \mathrm{Tr}\bigg[\frac{(\sigma\sigma^\top\Delta v^n)(t,x)}2\bigg] + \frac{\|\sigma^\top\nabla v^n(t,x)\|^2}{2\varphi(n)}+ h\big(t, x, (\sigma^\top\nabla v^n)(t,x),u^\star(t,x,(\sigma^\top\nabla v^{n-1})(t,x))\big) =0,\; (t,x)\in[0,T)\times\R^m,\\[0.8em]
	v^n(T,x) = G(x),\; x\in\R^m.
	\end{cases}
\end{equation*}
Notice that the above PDE is not linear, but it can be linearised using the standard Cole--Hopf exponential transform since $h$ is linear in its third argument, and has a purely quadratic non-linearity in the gradient. More precisely, defining $u^n\coloneqq \exp(v^n/\varphi(n))$, $n\in\N$, we have
\begin{equation*}
	\begin{cases}
	\displaystyle u^n_t(t,x)+ \frac12\mathrm{Tr}\big[(\sigma\sigma^\top\Delta u^n)(t,x)\big] +\sigma(t,x) b\big(t,x,u^\star(t,x,(\sigma^\top\nabla v^{n-1})(t,x))\big)\cdot \nabla u^n(t,x)\\
	\displaystyle\quad- \frac{L\big(t, x, u^\star(t,x,(\sigma^\top\nabla v^{n-1})(t,x))\big)}{\varphi(n)}u^n(t,x) =0,\; (t,x)\in[0,T)\times\R^m,\\[0.8em]
	u^n(T,x) = \exp(G(x)/\varphi(n)),\; x\in\R^m.
	\end{cases}
	\end{equation*}
	That is
	\begin{equation}
\label{eq:def.vn}
	 v^n(t,x) = \varphi(n)\log\bigg(\E^{\PP^\smalltext{n}}\bigg[\exp\bigg(\frac{F^n(X_{T})}{\varphi(n)}\bigg)\bigg] \bigg),\; t\in[0,T],\; \text{with}\; \frac{\mathrm{d}\PP^{n}}{\mathrm{d}\PP}\coloneqq \cE\bigg(\int_t^\cdot  b\big(s,X_{s},u^\star(s, X_{s}, (\sigma^\top\nabla v^{n-1})(s,X_s))\big)\cdot \mathrm{d} B_s\bigg)_T,
\end{equation}
with
	\begin{equation*}
	F^n(X_{T}) \coloneqq G(X_{ T}) - \int_t^TL\big(s,X_{ s},u^\star(s, X_{s}, (\sigma^\top\nabla v^{n-1})(s,X_s))\big)\mathrm{d}s,\;\text{and } X_t =x.
\end{equation*}

Furthermore, we also know from \Cref{prop:Bellman} and {\rm \citeauthor*{ma1999forward} \cite{ma1999forward}} that the value $V$ of the problem \eqref{eq:OG.problem} satisfies $V = v(0,x_0),$ where $v$ is a viscosity solution of the HJB equation associated to the control problem \eqref{eq:OG.problem}
\begin{equation}
\label{eq:HJB}
	v_t(t,x) + \frac12\mathrm{Tr}\big[(\sigma \sigma^\top\Delta v)(t,x)\big] + H\big(t,x,(\sigma^\top\nabla v)(t,x)\big) = 0,\; v(T,x) = G(x),\; (t,x)\in[0,T)\times\R^m.
\end{equation}
The following is a direct corollary of our main result.
\begin{corollary}
\label{cor:HJB}
	Let {\rm\Cref{assump:standing}} be satisfied, and further assume that the functions $b$, $G$, $L$ and $\sigma$ are bounded twice continuously differentiable in $x$, continuous in $t$ and with bounded derivatives, and $\sigma$ uniformly non-generate
	Then we have for some $C>0$
	\begin{equation*}
		-\frac{C}{2^n}\le v^n(t,x) - v(t,x)\le \frac{C}{\varphi(n)},\; \forall (n,t,x)\in \N\times [0,T]\times \R^m.
	\end{equation*}
	Moreover, the control
	$\alpha^\star_t\coloneqq u^\star(t, X_{t}^{\alpha^\smalltext{\star}}, (\sigma\nabla v)(t, X_{t}^{\alpha^\smalltext{\star}}))$ is a solution of the stochastic optimal control problem \ref{eq:OG.problem}, where $ X^{\alpha^\smalltext{\star}}$ is the unique solution of the {\rm SDE}
	\begin{equation}
	\label{eq:Markov.SDE.strong}
		X_t^{\alpha^\smalltext{\star}} = x+ \int_0^t\sigma\big(s,X^{\alpha^\smalltext{\star}}_{ s},\phi(t, X_{s}^{\alpha^\smalltext{\star}}))\big)b\big(s,X^{\alpha^\smalltext{\star}}_{s}, \phi(s,X_{s}^{\alpha^\smalltext{\star}})\big)\mathrm{d}s + \int_0^t\sigma\big(s,X^{\alpha^\smalltext{\star}}_{ s},\phi(s,X_{s}^{\alpha^\smalltext{\star}})\big)\mathrm{d}B_s, t\in[0,T],\; \P\text{\rm--a.s.}
	\end{equation}
	with $\phi(t,x) \coloneqq u^\star(t, x, (\sigma\nabla v)(t,x))$ and there is $C>0$ such that
	\begin{equation*}
		\E^\P\bigg[\sup_{t\in [0,T]}\|X^n_t - X_t^{\alpha^\smalltext{\star}} \| + \int_0^T\big\|
		\alpha^{\star,n}_t- \alpha^\star_t\big\|\mathrm{d}t\bigg]\le \frac{C}{2^{n}},\; \forall n\in\N
	\end{equation*}
	with $\alpha^{\star,n}_t \coloneqq u^\star\big(t, X^n_t, (\sigma^\top\nabla v^n)(t, X^n_t)\big),$ where for each $n\in \N$ the process $X^n$ is the unique strong solution of {\rm SDE} \eqref{eq:Markov.SDE.strong} with $v$ therein replaced by $v^n$.
\end{corollary}
\begin{remark}
	The functions $\nabla v$ and $\nabla v^n$ are bounded and Lipschitz-continuous for instance when the functions $b$, $G$, $L$ and $\sigma$ are bounded twice continuously differentiable in $x$, continuous in $t$ and with bounded derivatives, and $\sigma$ uniformly non-generate see \emph{e.g.} {\rm\citeauthor*{ma2024convergence} \cite[Lemma 2.2]{ma2024convergence}}.
	In particular, the SDE \eqref{eq:Markov.SDE.strong} is well posed in the strong sense.
	These properties will be used in the proof of {\rm\Cref{cor:HJB}}.
\end{remark}
We stress that under the Lipschitz continuity condition of $\nabla v$, we obtain a rate of convergence for the approximation of the \emph{strict} optimal control, which in addition is in feedback form.
In the literature, see {\rm \citeauthor*{huang2024convergence} \cite{huang2024convergence}} and {\rm \citeauthor*{ma2024convergence} \cite{ma2024convergence}}, only approximation of the optimal \emph{relaxed} control has been obtained.

\subsection{The case of controlled volatility}

In this subsection let us assume that the function $\sigma$ is allowed to depend on the control variable $u$.
We will explain how our approximation scheme extends to the present case.
The scheme is essentially the same, we present it here again for sake of clarity.
Let us introduce the---full---Hamiltonian for $(t,\x,z,\gamma,\Sigma)\in[0,T]\times\Cc_m\times\S^m_{+}\times\S^m_{+}$, with $\S^m_+$ denoting the space of $m\times m$ symmetric, positive semi-definite matrices. 
\begin{equation*}
	\cH(t,\x,z,\gamma)\coloneqq \sup_{\Sigma\in\S^\smalltext{m}_\smalltext{+}}\bigg\{ \frac12\mathrm{Tr}[\Sigma\gamma] + H(t, \x,z, \Sigma) \bigg\},\; \text{with}\;  H(t, \x,z, \Sigma) \coloneqq  \sup_{\{u\in U: \sigma\sigma^{\smalltext{\top}}(t,x,u)=\Sigma\}}\big\{ \sigma(t,\x,u)b(t,\x,u)\cdot z - L(t, \x,u)\big\}.
\end{equation*}
	We will make the following assumptions.
\begin{assumption}\label{Ass:Control.Vol}
			$(i)$ {\rm\Cref{assump:standing}}.$(i)$--$(ii)$ is satisfied, the functions 
		\begin{align*}
			&\Cc_m\times U\times\R^d\ni(\x,u,z) \longmapsto \sigma(t,\x,u)b(t,\x,u), 
			\;\Cc_m\times U\ni(\x,u)\longmapsto \sigma(t,\x,u)\in\R^{m\times d},
			 \;  \Cc_m\ni\x\longmapsto G(\x)\in\R,
		\end{align*}  
		are Lipschitz-continuous, uniformly in $t$,
		and we denote by $\ell_{\sigma_\smalltext{u}}$ and $\ell_{G}$ the respective Lipschitz-continuity constants of $\sigma$ and $G$ with respect to $u$ and $\x;$
		\medskip
		$(ii)$ any function  $u^\star$ such that 
		\begin{equation}
		\label{eq:argmax.pro.Vol}
			u^\star(t,\x,z,\Sigma) \in  \underset{\{u\in U: \sigma\sigma^{\smalltext{\top}}(t,x,u)=\Sigma\}}{\mathrm{argmax}}\big\{ \sigma(t,\x,u)b(t,\x,u)\cdot z - L(t, \x,u)\big\},\; (t,\x,z,\Sigma)\in[0,T]\times\Cc_m\times\S^m_{+},
	\end{equation}
	is Borel-measurable and $\Cc_m\times\R^d\ni(\x,z)\longmapsto u^\star(t,\xi,z,\Sigma)\in U$, is Lipschitz continuous.
	We denote by  $\ell_{u^{\smalltext{\star}}_\smalltext{x}}$ the Lipschitz-continuity constant of $u^\star$ in $\x$.
\end{assumption}
We are now given a Borel-measurable function $u^\star:[0,T]\times\Cc_m\times\S^m_{+}\longrightarrow \cU $ from \Cref{Ass:Control.Vol}.$(ii)$, as well as a non-decreasing function $\varphi:\R_+ \longrightarrow \R_+^\star$ such that $\varphi(0) = 1$.
On some probability space $(\Omega^\star,\cF^\star,\PP^\star)$ with Brownian motion $B^\star$ possibly different from the underlying probability space and Brownian motion, we define the following scheme:
put $Z^{-1}=0$ and for all $n\in \N$ consider the control problem
\begin{equation}
\label{eq:problem.n.vol.control}
	\begin{cases}
		\displaystyle V^n \coloneqq  \sup_{\alpha\in \cA}\E^{\PP^{\smalltext{\alpha}}}\bigg[G(X^n) - \int_0^T\bigg(\frac12\|\alpha_s\|^2 + L\big(s,X^n_{\cdot\wedge s}, u^\star(s, X^n_{\cdot\wedge s}, Z^{n-1}_s)\big)\d s \bigg],\\[0.8em]
		\displaystyle\frac{\mathrm{d}\PP^{\alpha}}{\mathrm{d}\PP^\star}\coloneqq \cE\bigg(\int_0^\cdot \bigg( b\big(s,X^n_{\cdot\wedge s},u^\star(s, X^n_{\cdot\wedge s}, Z^{n-1}_s)\big) + \frac{\alpha_s}{\varphi^{1/2}(n)}\bigg)\cdot \d B_s^\star\bigg)_T,
	\end{cases}
\end{equation}
where $(Y^n,Z^n)$ is the unique solution of the BSDE
\begin{equation}
	\label{eq:BSDE.n.vol}
		Y^n_t = F^n(X^n) + \int_t^T\bigg(\frac{1}{2\varphi(n)}\|Z_s^n\|^2+ b\big(s,X^n_{\cdot\wedge s},u^\star(s, X^n_{\cdot\wedge s}, Z^{n-1}_s)\big)\cdot Z^n_s\bigg)\d s - \int_t^TZ_s^n\cdot \d B_s^\star, \; t\in [0,T], \; \PP^\star\text{\rm--a.s.},
	\end{equation}
	with
	\begin{equation*}
		F^n(X^n) \coloneqq G(X^n) - \int_0^TL\big(s,X^n_{\cdot\wedge s},u^\star(s, X^n_{\cdot\wedge s}, Z^{n-1}_s)\big)\d s,
	\end{equation*}
and $X^n$ satisfying 
\begin{equation}
\label{eq:iterated.SDE.vol}
	X^n_t = \xi + \int_0^t\sigma\big(s,X^n_{\cdot\wedge s}, u^\star(s, X^n_{\cdot\wedge s}, Z^{n-1}_s)\big)\d B^\star_s,\; t\in[0,T],\; \P^\star\text{\rm--a.s.}
\end{equation}
We will say that a control process $\nu$ is of feedback form if it is a function of the state process.
That is, $\nu:[0,T]\times \cC_m\longrightarrow U$ and the SDE
\begin{equation*}
	X^\nu_t =\xi+ \int_0^t\sigma\big(s,X^\nu_{\cdot\wedge s},\nu_s(X^\nu_{\cdot\wedge s})\big)b\big(s,X_{\cdot\wedge s}^\nu, \nu_s(X^\nu_{\cdot\wedge s})\big)\mathrm{d}s + \int_0^t\sigma\big(s,X^\nu_{\cdot\wedge s},\nu_s(X^\nu_{\cdot\wedge s})\big)\mathrm{d}B_s, t\in[0,T],\; \P\text{\rm--a.s.}
\end{equation*}
admits a unique strong solution.
\begin{theorem}
\label{thm:Rate.vol}
	Let {\rm\Cref{ass.z.bounded}} and {\rm\Cref{Ass:Control.Vol}} be satisfied, and let $\nu^\star$ be an optimal control in feedback form for the control problem \eqref{eq:OG.problem}.
	Then, there exist a function $u^\star$ satisfying \eqref{eq:argmax.pro.Vol}, a probability space $(\Omega^\star,\cF^\star,\PP^\star)$ along with a Brownian motion $B^\star$ defined on it such that the scheme $(V^n,X^n, Y^n, Z^n,\nu^n)_{n\in\N}$ is well-defined, and it holds
	\begin{equation}
	\label{eq:Def.Yn.vol}
		V^n= \varphi(n)\log\bigg(\E^{\PP^\smalltext{n}}\bigg[\exp\bigg(\frac{1}{\varphi(n)}F^n(X^n)\bigg)\bigg] \bigg),\; \text{\rm with}\; \frac{\mathrm{d}\PP^{n}}{\mathrm{d}\PP^\star}\coloneqq \cE\bigg(\int_0^\cdot b\big(s,X^n_{\cdot\wedge s},u^\star(s, X^n_{\cdot\wedge s}, Z^{n-1}_s)\big)\cdot \d B_s^\star\bigg)_T.
	\end{equation}
	If we have 
	\[
	 8\ell^2_G\ell_{\sigma_u}^2\ell_{u^\smalltext{\star}_\x}^2\mathrm{e}^{ 8(\ell_{\smalltext{\sigma}_\smalltext{\x}}+\ell_{\smalltext{\sigma_u}}\ell_{\smalltext{u}^\tinytext{\star}_\smalltext{\x}})^\smalltext{2}T}<1,
	\] then there is some constant $C>0$ such that for all $n\in \N$.
	\begin{equation}
		| V^n- V | 
		\le \frac{C}{2^{n}}.
	\end{equation}
\end{theorem}
This result shows that with small modifications, the same scheme introduced above allows to approximate the value of the volatility-controlled problem, provided that additional regularity is assumed on the coefficients in the state process.
We also need the Lipschitz-continuity constants of one of the functions $G$, $\sigma(t,\x,\cdot)$ or $u^\star(t,\cdot,z,\Sigma)$ to be small enough.
For instance, if $\sigma$ does not depend on $u$, then $\ell_{\sigma_\smalltext{u}} = 0$ and we are back in the uncontrolled volatility case.
If the terminal reward is zero, then we can allow the other constants to be arbitrary.
Moreover, when the functions $\sigma b$ and $L$ are separated into sums of functions of $(t,\x)$ and $(t,u)$, it follows that $u^\star$ does not depend on $\x$ and then $\ell_{u_\x} = 0$, in which case our condition is automatically satisfied.

\medskip

The scheme \eqref{eq:problem.n.vol.control}--\eqref{eq:iterated.SDE.vol} differs from the case of un-controlled volatility only from the fact that at each step we need to solve the \emph{forward} SDE \eqref{eq:iterated.SDE.vol}.
Given $Z^{n-1}$, this is efficiently solved using well-understood numerical schemes, see \emph{e.g.} \citeauthor{kloeden1992numerical} \cite{kloeden1992numerical}, and simulating the processes $(Y^n,Z^n)$ amount to solving a linear BSDE as in \Cref{sec:no.vol}.

\medskip

To the best of our knowledge, only the papers by {\rm \citeauthor*{tran2024policy} \cite{tran2024policy}} and by {\rm \citeauthor*{ma2024convergence} \cite{ma2024convergence}} treat the case of diffusion control.
In {\rm \citeauthor*{tran2024policy} \cite{tran2024policy}, the approximation of both the value and the optimal control are obtained under smallness conditions, and in {\rm \citeauthor*{ma2024convergence} \cite{ma2024convergence}} the approximation of both the optimal controls and the value are obtained in the case $m=d=1$, and this restriction is fundamental for their proof method.
Both of these papers derive the same rate as ours.
In the present paper we deal with the general multi-dimensional case.
As in {\rm \citeauthor*{ma2024convergence} \cite{ma2024convergence}}, we obtain the approximation of the value of the problem, though we do not really have access to optimal controls here.

\begin{remark}
It is notable that {\rm\cite{ma2024convergence} } treats both the finite and (discounted) infinite horizon versions of the optimal control problem. We have decided to concentrate here on the finite time horizon, but in principle our algorithm, and the associated proofs of convergence would extend \emph{mutatis mutandis}. The point being that optimal control problem with discounted infinite horizon---or even random horizon---can be associated with {\rm BSDEs} and {\rm $2$BSDEs} with infinite/random horizon, whose well-posedness and the corresponding estimates have for instance been established in {\rm\citeauthor*{papapantoleon2016existence} \cite{papapantoleon2016existence}}, {\rm\citeauthor*{possamai2024reflections} \cite{possamai2024reflections}}, {\rm\citeauthor*{lin2020second} \cite{lin2020second}}, and {\rm\citeauthor*{possamai2024second} \cite{possamai2024second}.}
\end{remark}

\section{Proofs}
\label{sec:proofs}

\subsection{Proof of Theorem \ref{thm:Rate}}
In this subsection we assume that the assumptions of \Cref{thm:Rate} are in force.
We will constantly use the so-called weak formulation of the above control problem, which is defined as
\begin{equation}
\label{eq:weak.control}
	V_{\mathrm{w}} \coloneqq  \sup_{\nu \in \Uc}\E^{\P^\smalltext{\nu}}\bigg[G(X) - \int_0^TL\big(s, X_{\cdot\wedge s}, \nu_s\big)\d s \bigg],\; \text{with}\; \frac{\mathrm{d}\PP^{\nu}}{\mathrm{d}\PP}\coloneqq \cE\bigg(\int_0^\cdot b(s,X_{\cdot\wedge s},\nu_s)\cdot \d B_s\bigg)_T
\end{equation}
where $X$ is the solution of the driftless SDE \eqref{eq:driftless}. 
It is well-known that under mild conditions, the values of the weak and strong formulations of the control problem coincide.
That is, $V= V_w$,
see for instance \citeauthor*{karoui2013capacities2} \cite[Theorem 4.5]{karoui2013capacities2}. 
Now, let $u^\star$ be a function given by Assumption \ref{assump:standing}$(v)$ and consider the solution $(Y,Z)$ of the BSDE \eqref{bsde}.
By the Bellman optimality principle (see \Cref{prop:Bellman}) the stochastic process
$\nu^\star$ given by
\begin{equation}
\label{eq:charac.nustar}
	\nu^\star_t = u^\star(t, X_{\cdot\wedge t}, Z_t) ,\; \d t\otimes \P\text{\rm --a.e.},
\end{equation}
is optimal for the control problem \eqref{eq:weak.control} and it holds $V_w = Y_0$.
Therefore, we have $V = Y_0$ and thus it suffices to approximate $Y_0$, which is the goal of the rest of the proof.

\medskip

Recall that the approximating BSDEs are given by 
\begin{equation}
\label{eq:BSDE.nn}
	\widehat Y^n_t = F^n(X) + \int_t^T\bigg(\frac{1}{2\varphi(n)}\|Z_s^n\|^2+ b\big(s,X_{\cdot\wedge s},u^\star(s, X_{\cdot\wedge s}, Z^{n-1}_s)\big)\cdot Z^n_s\bigg)\mathrm{d}s - \int_t^TZ_s^n\cdot \mathrm{d}B_s, \; t\in [0,T], \; \PP\text{\rm--a.s.},
\end{equation}
 for any $n\in\N$, with the convention $Z^{-1}\coloneqq 0$, with
 \begin{equation*}
	F^n(X) \coloneqq G(X) - \int_0^TL\big(s,X_{\cdot\wedge s},u^\star(s, X_{\cdot\wedge s}, Z^{n-1}_s)\big)\mathrm{d}s.
\end{equation*}
Using the change of variables
\begin{equation*}
	Y^n_t\coloneqq \widehat Y^n_t + \int_0^tL\big(s,X_{\cdot\wedge s},u^\star(s, X_{\cdot\wedge s}, Z^{n-1}_s)\big)\mathrm{d}s \quad t\in [0,T],
\end{equation*}
we have that $(Y^n,Z^n)$ satisfies the equation
\begin{equation}\label{bsden}
	Y_t^n=G(X)+\int_t^T\bigg(h\big(s,X_{\cdot\wedge s},Z^n_s,u^\star(s,X_{\cdot\wedge s},Z_s^{n-1})\big)+\frac{\|Z_s^n\|^2}{2\varphi(n)}\bigg)\mathrm{d}s-\int_t^TZ^n_s\cdot\mathrm{d}B_s,\; t\in[0,T),\; \P\text{\rm--a.s.}
\end{equation}
Notice that we have $V^n = Y^n_0$.
Our first result is standard for quadratic BSDEs, and provides uniform \emph{a priori} estimates for the solutions to BSDEs \eqref{bsde} and \eqref{bsden}.
\begin{lemma}\label{lemma:unif}
For any $p>1$, we have that there is some constant $C_p$ depending only on $p$ such that for any $n\in\N$
\[
	\E^\P\bigg[\exp\bigg(p\sup_{t\in[0,T]}|Y_t^n|\bigg)+\exp\bigg(p\sup_{t\in[0,T]}|Y_t|\bigg)\bigg]+\E^\P\bigg[\bigg(\int_0^T\|Z_s^n\|^2\mathrm{d}s\bigg)^{\frac p2}\bigg]\leq C_p \E^\P\bigg[\exp\bigg(p\bigg(|G(X)|+\int_0^T|L(s,X_{\cdot\wedge s},\nu_s)|\mathrm{d}s\bigg)\bigg].
\]
In particular, for any $p>1$, there is some $C>0$ such that
\[
\sup_{n\in\N}\bigg\{\E^\P\bigg[\exp\bigg(p\sup_{t\in[0,T]}|Y_t^n|\bigg)\bigg]\bigg\}+\sup_{n\in\N}\bigg\{\E^\P\bigg[\bigg(\int_0^T\|Z_s^n\|^2\mathrm{d}s\bigg)^{\frac p2}\bigg]\bigg\}\leq C.
\]
\end{lemma}

\begin{proof}
The first two inequalities are direct from \citeauthor*{briand2008quadratic} \cite[Equation (5)]{briand2008quadratic} and from \cite[Corollary 4]{briand2008quadratic}. In particular, one should realise that since $\varphi$ increases to $+\infty$, we can select a value of $\gamma$ in the notations of \cite{briand2008quadratic} which is independent of $n$ in the estimates for \Cref{bsden}. Similarly, since the generator of \Cref{bsde} is actually Lipschitz-continuous, it satisfies the same quadratic growth assumption for any choice of $\gamma$. The second result now stems from the fact that since \Cref{assump:standing}.$(iii)$ holds, we can use \Cref{lem:estimX} to ensure that the right-hand side in the previous estimates is finite.
\end{proof}

Toward our goal to show the convergence, there is one inequality which fully uses the structure of the approximation we have chosen.
\begin{proposition}
\label{prop:upperbound}
There is some $C>0$ such that for any $n\in\N$, we have
\[
Y^n_0-Y_0\leq \frac{C}{\varphi(n)}.
\]
\end{proposition}
\begin{proof}
We compute directly that for any $n\in\N$
\begin{align*}
Y^n_0-Y_0&=\int_0^T\bigg(h\big(s,X_{\cdot\wedge s},Z^n_s, u^\star(s,X_{\cdot\wedge s},Z_s^{n-1})\big)-H(s,X_{\cdot\wedge s},Z_s)+\frac{\|Z^n_s\|^2}{2\varphi(n)}\bigg)\mathrm{d}s-\int_0^T(Z^n_s-Z_s)\cdot\mathrm{d}B_s\\
&\leq \int_0^T\bigg(H(s,X_{\cdot\wedge s},Z^n_s)-H(s,X_{\cdot\wedge s},Z_s)+\frac{\|Z^n_s\|^2}{2\varphi(n)}\bigg)\mathrm{d}s-\int_0^T(Z^n_s-Z_s)\cdot\mathrm{d}B_s\\
&=\frac1{2\varphi(n)}\int_0^T\|Z^n_s\|^2\mathrm{d}s-\int_0^T(Z^n_s-Z_s)\cdot\big(\mathrm{d}B_s-\lambda_s\mathrm{d}s\big),
\end{align*}
where we used successively the definition of $H$ as a supremum, and that since $H$ is Lipschitz-continuous in $z$, there exists a bounded, $\R^d$-valued, $\F$--progressively measurable process $\lambda$ such that
\[
H(s,X_{\cdot\wedge s},Z^n_s)-H(s,X_{\cdot\wedge s},Z_s)=\lambda_s\cdot(Z^n_s-Z_s),\; \mathrm{d}s\otimes\P\text{\rm --a.e.}
\]
Consider then the probability measure $\Q$ on $(\Omega,\Fc_T)$ with density $\Ec\big(\int_0^\cdot\lambda_s\cdot\mathrm{d}B_s\big)_T$ with respect to $\P$. We deduce using the fact that $\lambda$ is bounded, Cauchy--Schwarz's inequality, and \Cref{lemma:unif}, that there is some constant $C>0$ (which can change value from line to line) such that
\begin{align*}
Y^n_0-Y_0&\leq \frac1{2\varphi(n)}\E^\Q\bigg[\int_0^T\|Z^n_s\|^2\mathrm{d}s\bigg]\leq \frac{C}{\varphi(n)}\E^\P\bigg[\bigg(\int_0^T\|Z^n_s\|^2\mathrm{d}s\bigg)^2\bigg]^{1/2}\leq \frac{C}{\varphi(n)}.
\end{align*}
\end{proof}
In order to get the reverse inequality, we need extra estimates.
\begin{proposition}
\label{prop:lowerbound}
	Assume that {\color{black}$\varphi$ grows super-exponentially to $+\infty$}. Then for any $(\eta,\eps) \in(0,1)^2$, there is some $C>0$ such that for any $n\in\N$ and 
\[
(Y_0^n-Y_0)^2+\E^\P\bigg[\bigg(\int_0^T\|Z^n_s-Z_s\|^2\d s\bigg)^{1-\eps}\bigg]\leq C\eta^n.
\]
\end{proposition}

\begin{proof}
Given some $\beta>0$, which we will fix later, applying It\^o's formula leads us to
\begin{align*}
&(Y_0^n-Y_0)^2+\beta\int_0^T\mathrm{e}^{\beta s}(Y_s^n-Y_s)^2\mathrm{d}s+\int_0^T\mathrm{e}^{\beta s}\|Z_s^n-Z_s\|^2\mathrm{d}s\\
&=2\int_0^T\mathrm{e}^{\beta s}(Y_s^n-Y_s)\bigg(h\big(s,X_{\cdot\wedge s},Z^n_s,u^\star(s,X_{\cdot\wedge s},Z_s^{n-1})\big)-H(s,X_{\cdot\wedge s},Z_s)+\frac{\|Z^n_s\|^2}{2\varphi(n)}\bigg)\mathrm{d}s-2\int_0^T\mathrm{e}^{\beta s}(Y_s^n-Y_s)(Z_s^n-Z_s)\cdot\mathrm{d}B_s\\
&=2\int_0^T\mathrm{e}^{\beta s}(Y_s^n-Y_s)\big(b\big(s,X_{\cdot\wedge s},u^\star(s,X_{\cdot\wedge s},Z_s^{n-1})\big)-b\big(s,X_{\cdot\wedge s},u^\star(s,X_{\cdot\wedge s},Z_s)\big)\big)\cdot Z_s\mathrm{d}s\\
&\quad+2\int_0^T\mathrm{e}^{\beta s}(Y_s^n-Y_s)\bigg(L\big(s,X_{\cdot\wedge s},u^\star(s,X_{\cdot\wedge s},Z_s^{n-1})\big)-L\big(s,X_{\cdot\wedge s},u^\star(s,X_{\cdot\wedge s},Z_s)\big)+\frac{\|Z^n_s\|^2}{2\varphi(n)}\bigg)\mathrm{d}s
\\
&\quad-2\int_0^T\mathrm{e}^{\beta s}(Y_s^n-Y_s)(Z_s^n-Z_s)\cdot\big(\mathrm{d}B_s-b\big(s,X_{\cdot\wedge s},u^\star(s,X_{\cdot\wedge s},Z_s^{n-1})\big)\mathrm{d}s\big)\\
&\leq \frac{\ell_{u^\smalltext{\star}}^2(\ell_L+\ell_b\|Z\|_\infty)^2}{\eta^2}\int_0^T\mathrm{e}^{\beta s}(Y_s^n-Y_s)^2\mathrm{d}s+\eta^2\int_0^T\mathrm{e}^{\beta s}(Z_s^{n-1}-Z_s)^2\mathrm{d}s+\frac1{\varphi(n)}\sup_{t\in[0,T]}|Y^n_t-Y_t|\int_0^T\mathrm{e}^{\beta s}\|Z^n_s\|^2\mathrm{d}s\\
&\quad-2\int_0^T\mathrm{e}^{\beta s}(Y_s^n-Y_s)(Z_s^n-Z_s)\cdot\big(\mathrm{d}B_s-b\big(s,X_{\cdot\wedge s},u^\star(s,X_{\cdot\wedge s},Z_s^{n-1})\big)\mathrm{d}s\big).
\end{align*}
Since $b$ is bounded, we can define the probability measure $\Q$ on $(\Omega,\Fc_T)$ with density $\Ec\big(\int_0^\cdot b\big(s,X_{\cdot\wedge s},u^\star(s,X_{\cdot\wedge s},Z_s^{n-1})\big)\cdot\mathrm{d}B_s\big)_T$ with respect to $\P$, and choosing $\beta \geq \ell_{u^\smalltext{\star}}^2(\ell_L+\ell_b\|Z\|_\infty)^2/\eta^2$, we deduce\footnote{The stochastic integral is a true martingale by standard estimates and the integrability proved in \Cref{lemma:unif}.}
\begin{align*}
(Y_0^n-Y_0)^2+\E^\Q\bigg[\int_0^T\mathrm{e}^{\beta s}\|Z_s^n-Z_s\|^2\mathrm{d}s\bigg]&\leq \eta^2\E^\Q\bigg[\int_0^T\mathrm{e}^{\beta s}\|Z_s^{n-1}-Z_s\|^2\mathrm{d}s\bigg]+\frac1{\varphi(n)}\E^\Q\bigg[\sup_{t\in[0,T]}|Y^n_t-Y_t|\int_0^T\mathrm{e}^{\beta s}\|Z^n_s\|^2\mathrm{d}s\bigg]\\
&\leq\eta^2\E^\Q\bigg[\int_0^T\mathrm{e}^{\beta s}\|Z_s^{n-1}-Z_s\|^2\mathrm{d}s\bigg]\\
&\quad+\frac1{\varphi(n)}\E^\Q\bigg[\sup_{t\in[0,T]}|Y^n_t-Y_t|^{1+\eps}\bigg]^{\frac1{1+\eps}}\E^\Q\bigg[\bigg(\int_0^T\mathrm{e}^{\beta s}\|Z^n_s\|^2\mathrm{d}s\bigg)^{\frac{1+\eps}\eps}\bigg]^{\frac\eps{1+\eps}}\\
&\leq \eta^2\E^\Q\bigg[\int_0^T\mathrm{e}^{\beta s}\|Z_s^{n-1}-Z_s\|^2\mathrm{d}s\bigg]+\frac{C}{\varphi(n)}.
\end{align*}
Defining
\[
u_n\coloneqq (Y_0^n-Y_0)^2+\E^\Q\bigg[\int_0^T\mathrm{e}^{\beta s}\|Z_s^n-Z_s\|^2\mathrm{d}s\bigg],
\]
we thus obtained from an immediate recursion that for any $n\in\N^\star$
\[
u_n\leq \eta^{2n} u_0+C\sum_{k=0}^{n-1}\frac{\eta^{2k}}{\varphi(n-k)},\; n\in\N.
\]
Now, using the fact that $\varphi$ is non-decreasing, we have
\begin{align*}
u_n\leq \eta^{2n} u_0+C\sum_{k=0}^{\lfloor \frac{\smalltext{n}\smalltext{-}\smalltext{1}}{\smalltext{2}} \rfloor}\frac{\eta^{2k}}{\varphi(n-k)}+C\sum_{k=\lfloor \frac{\smalltext{n}\smalltext{-}\smalltext{1}}{\smalltext{2}} \rfloor+1}^{n-1}\frac{\eta^{2k}}{\varphi(n-k)}\leq \eta^{2n} u_0+\frac{C}{(1-\eta^2)\varphi(n-\lfloor \frac{n-1}{2}\rfloor)}+\frac{C}{(1-\eta^2)\varphi(1)}\big(\eta^{2\lfloor \frac{\smalltext{n}\smalltext{-}\smalltext{1}}{\smalltext{2}}\rfloor+2}-\eta^{2n}\big).
\end{align*}
Given our assumption on the growth of $\varphi$, this gives the desired result using H\"older's inequality to move from a bound under $\Q$ to a bound under $\P$.
\end{proof}

\begin{proof}[Proof of \Cref{thm:Rate}]
	We can now finish the proof of \Cref{thm:Rate}.
	First of all, by construction and subsequent applications of \Cref{prop:Bellman}, it follows that for all $n\in \N$ we have $V^n = \widehat Y^n_0 = Y^n_0$.
	We have already argued that $V = Y_0$.
	Therefore, it follows by \Cref{prop:lowerbound} and \Cref{prop:upperbound} that $-\frac{C}{2^{n}}\le V^n- V \le \frac{C}{\varphi(n)}.$
\end{proof}

\subsection{Approximation of optimal controls}
\label{sec.pointwise.conv.optim.cont}
In this section we turn our attention to the approximation of an optimal control.
It follows directly from the proof of \Cref{thm:Rate}.
In fact, we have:
\begin{proposition}
\label{prop:Norm.conv}
	Let {\rm\Cref{assump:standing}} and {\rm\Cref{ass.z.bounded}} be satisfied.
	Then the stochastic optimal control problem \eqref{eq:weak.control} admits a solution $\nu^\star$, and for any $\eps\in(0,1]$, there is $C>0$ such that
	\begin{equation*}
		\E^\P\bigg[\bigg(\int_0^T\big\|u^\star(t, X_{\cdot\wedge t},Z^n_t) - \nu^\star_t\big\|^2\mathrm{d}t\bigg)^{1-\eps}\bigg]\le \frac{C}{2^{n}},\; 
			\forall n\in\N.
	\end{equation*}
\end{proposition}
\begin{proof}
	With the notation of the proof of \Cref{thm:Rate}, by \Cref{eq:charac.nustar}, boundedness of $\cU$ and \Cref{assump:standing}.$(v)$, we have
	\begin{align*}
		\E^\P\bigg[\bigg(\int_0^T\big\|u^\star(t, X_{\cdot\wedge t},Z^n_t) - \nu^\star_t\big\|^2\mathrm{d}t\bigg)^{1-\eps}\bigg] &\le C\E^\P\bigg[\bigg(\int_0^T\big\|u^\star(t, X_{\cdot\wedge t},Z^n_t) - u^\star(t, X_{\cdot\wedge t}, Z_t)\big\|^2\mathrm{d}t\bigg)^{1-\eps}\bigg]\\
		&\le C\E^\P\bigg[\bigg(\int_0^T\|Z^n_t - Z_t\|^2\mathrm{d}t\bigg)^{1-\eps}\bigg].
	\end{align*}
	We conclude using \Cref{prop:lowerbound}.
\end{proof}
It is noteworthy that the rate obtained in \Cref{prop:Norm.conv} does not assume uniqueness of the optimal control $\nu^\star$ of Problem \eqref{eq:weak.control}.
As observed in the proof, the main idea of this result is that the approximating controls $ \nu^{\star,n}_t\coloneqq u^\star(t, X_{\cdot\wedge t},Z^n_t)$ are constructed based on the choice of the targeted optimal control $\nu^{\star}$.
In other terms, there is a (unique) optimal control that we are approximating using the whole sequence $(\nu^{\star,n})_{n\in\N}$. 


\medskip

The rate on the approximation of the optimal control in \Cref{prop:Norm.conv} is given with respect to an $\mathbb{L}^2$-norm over the time argument.
When performing numerical simulations one might be interested in simulating the optimal control at a specific time $t$.
In this case, it would be interesting to assess the convergence rate at each time.
To this end, we will need to make stronger assumptions on the data.
\begin{assumption}\label{assumP:mall}
$(i)$ $h$ is Malliavin differentiable with a bounded Malliavin derivative such that in addition for any $i\in\{1,\dots,d\}$ and any $r\in[0,T]$, $D^i_rh(s,\x,z,u)$ is uniformly Lipschitz-continuous in $(z,u);$

\medskip
$(ii)$ $h$ is once continuously differentiable in $(z,u)$, with bounded derivatives which are also uniformly Lipschitz continuous in $(z,u);$

\medskip
$(iii)$ $u^\star$ is Malliavin differentiable with a bounded Malliavin derivative such that in addition for any $i\in\{1,\dots,d\}$ and any $r\in[0,T]$, $D^i_ru^\star(s,\x,z)$ is uniformly Lipschitz-continuous in $z;$

\medskip
$(iv)$ $u^\star$ is once continuously differentiable in $z$, with bounded derivatives which is also uniformly Lipschitz continuous in $z;$

\medskip
$(v)$ $G$ is bounded, Malliavin differentiable, with a bounded Malliavin derivative$;$

\medskip
$(vi)$ $L$ is uniformly bounded$;$

\end{assumption}
Under these assumptions, it can be checked, see e.g. \citeauthor*{mastrolia2014malliavin} \cite[Theorem 7.1]{mastrolia2014malliavin} that the solution $(Y,Z)$ of the BSDE \eqref{bsde} is Malliavin differentiable.
Since $X$ is Malliavin differentiable, it follows by regularity of $G$ and $L$ that $F^0(X)$ (recall \eqref{eq:BSDE0}).
Thus, by chain rule, $Y^0$ is Malliavin differentiable.
Then, it follows by \citeauthor*{pardoux1992backward} \cite[Lemma 2.3]{pardoux1992backward} that $Z^0$ is also Malliavin differentiable because it arises as the integrand in the martingale representation of a Malliavin differentiable martingale.
Using the same argument, we conclude by induction that for all $n\in \N$, the processes $(Y^n,Z^n)$ are Malliavin differentiable.
Let us differentiate both equations in the Malliavin sense. We have for any $i\in\{i,\dots d\}$, any $r\in[0,T]$ and any $n\in\N$
\begin{align*}
D_r^iY^n_t &= D_r^iG(X)+\int_t^T\Big(D_r^ih\big(s,X_{\cdot\wedge s},Z^n_s,u^\star(s,X_{\cdot\wedge s},Z_s^{n-1})\big)+\partial_zh\big(s,X_{\cdot\wedge s},Z^n_s,u^\star(s,X_{\cdot\wedge s},Z_s^{n-1})\big)D_r^iZ^n_s\Big)\mathrm{d}s\\
&\quad +\int_t^T\bigg(\partial_ah\big(s,X_{\cdot\wedge s},Z^n_s,u^\star(s,X_{\cdot\wedge s},Z_s^{n-1})\big)\big(D_r^iu^\star(s,X_{\cdot\wedge s},Z^{n-1}_s)+\partial _zu^\star(s,X_{\cdot\wedge s},Z^{n-1}_s)D_r^iZ^{n-1}_s\big)+\frac{Z^n_s\cdot D_r^iZ^n_s}{\varphi(n)}\bigg)\mathrm{d}s\\
&\quad-\int_t^TD^i_rZ^{n}_s\cdot\mathrm{d}B_s,\; t\in[0,T],
\end{align*}
and similarly
\begin{align*}
D_r^iY_t &= D_r^iG(X)+\int_t^T\Big(D_r^ih\big(s,X_{\cdot\wedge s},Z_s,u^\star(s,X_{\cdot\wedge s},Z_s)\big)+\partial_zh\big(s,X_{\cdot\wedge s},Z_s,u^\star(s,X_{\cdot\wedge s},Z_s)\big)D_r^iZ_s\Big)\mathrm{d}s\\
&\quad +\int_t^T\partial_ah\big(s,X_{\cdot\wedge s},Z_s,u^\star(s,X_{\cdot\wedge s},Z_s)\big)\big(D_r^iu^\star(s,X_{\cdot\wedge s},Z_s)+\partial _zu^\star(s,X_{\cdot\wedge s},Z_s)D_r^iZ_s\big)\mathrm{d}s-\int_t^TD^i_rZ_s\cdot\mathrm{d}B_s,\; t\in[0,T].
\end{align*}
Moreover, $Z^n$ and $Z$ are the respective traces of the Malliavin derivatives of $Y^n$ and $Z$.
That is, it holds
\begin{equation}
\label{eq:trace.DY}
	Z^{n,i}_t = D^i_tY^n_t,\; Z^i_t = D^i_tY_t\,\,\d\otimes \P\text{-a.s.,}
\end{equation}
where $Z^{n,i}$ is the $i^{\mathrm{th}}$ coordinate of $Z^n$.
We will also make the following condition:
\begin{assumption}
	\label{ass:DZ.bounded}
 $DZ$ is uniformly bounded$;$
\end{assumption}
Under Assumptions \ref{assumP:mall} and \ref{ass:DZ.bounded}, it is standard to prove the following.
\begin{lemma}\label{lemma:estimmall}
	Let {\rm Assumptions \ref{assumP:mall}} and {\rm\ref{ass:DZ.bounded}} hold. Then $DY$ is bounded, $Y^n$ and $DY^n$ are bounded uniformly in $n\in\N$, and $Z^n$ and $DZ^n$ are bounded in $\H^2_{\rm BMO}(\R^d,\F,\P)$, again uniformly in $n\in\N$. 
\end{lemma}
Using Assumptions \ref{assumP:mall} and \ref{ass:DZ.bounded}, and in particular that $DZ$, $Dh$ and $Du^\star$ are bounded, and that all the functions are Lipschitz-continuous, we get that there is some constant $C$ independent of $n$ such that
\begin{align*}
&|D_r^iY^n_t-D_r^iY_t|^2+\beta\int_t^T\mathrm{e}^{\beta s}|D_r^iY^n_s-D_r^iY_s|^2\mathrm{d}s+\int_t^T\mathrm{e}^{\beta s}\|D_r^iZ^n_s-D_r^iZ_s\|^2\mathrm{d}s\\
&\leq 2C\int_t^T\mathrm{e}^{\beta s}|D_r^iY^n_s-D_r^iY_s|\bigg(\|Z^{n-1}_s-Z_s\|+\|Z_s^n-Z_s\|+\|D^i_rZ^{n-1}_s-D^i_rZ_s\|+\frac{Z^n_s\cdot D_r^iZ_s}{\varphi(n)}\bigg)\mathrm{d}s\\
&\quad-2\int_t^T\mathrm{e}^{\beta s}(D_r^iY^n_s-D_r^iY_s)(D_r^iZ^n_s-D_r^iZ_s)\cdot\bigg(\mathrm{d}B_s-\underbrace{\bigg(\partial_zh\big(s,X_{\cdot\wedge s},Z_s^n,u^\star(s,X_{\cdot\wedge s},Z_s^{n-1})\big)+\frac{Z^n_s}{\varphi(n)}\bigg)}_{\eqqcolon \psi^\smalltext{n}_\smalltext{s}}\mathrm{d}s\bigg).
\end{align*}
Since by assumption the process $\psi^n$ is bounded, uniformly in $n\in\N$, in the space $\H^2_{\rm BMO}(\R^d,\F,\P)$, we can use it to define a probability measure $\Q$, equivalent to $\P$ and with finite second moment such that, using \Cref{lemma:estimmall} to ensure the required integrability properties, and taking $\beta$ large enough (independently of $n\in\N$)
\begin{align*}
&\E^\Q\big[|D_r^iY^n_t-D_r^iY_t|^2\big] +\E^\Q\bigg[\int_0^T\mathrm{e}^{\beta s}\|D_r^iZ^n_s-D_r^iZ_s\|^2\mathrm{d}s\bigg]\leq \frac{C}{\varphi(n)}+\frac{C}{2^n}+\eta\E^\Q\bigg[\int_0^T\|D^i_rZ^{n-1}_s-D^i_rZ_s\|^2\mathrm{d}s\bigg].
\end{align*}
Using \eqref{eq:trace.DY} and arguing as at the end of the proof of \Cref{prop:upperbound}, we can then easily proceed as above to deduce that for any $n\in\N$ and any $t\in[0,T]$
\begin{align}
\label{eq:bound.z.pointwise}
&\E^\Q\big[\|Z^n_t-Z_t\|^2\big]\leq \frac{C}{2^n}.
\end{align}
This allows to deduce the following result.
\begin{proposition}
\label{prop:pointwise.conv}
	Let {\rm Assumptions  \ref{assump:standing}}, {\rm\ref{ass.z.bounded}}, {\rm\ref{assumP:mall}} and {\rm\ref{ass:DZ.bounded}} be satisfied
	and let $\nu^\star$ be an optimal control for {\rm Problem \eqref{eq:weak.control}}.
	Then there exists a function $u^\star$ 
	satisfying \eqref{eq:argmax.ass} and
	such that, defining the scheme $(V^n,\widehat Y^n,Z^n,\alpha^n)_{n\in\N}$ through $u^\star$ and any $\varphi$, we have for some $C>0$ and almost every $t\in [0,T]$
	\begin{equation*}
		\E\big[ \big\|u^\star(t, X_{\cdot\wedge t},Z^n_t) - \nu^\star_t\big\| \big] \le \frac{C}{2^n}, \,\, n\in \N.
	\end{equation*}
\end{proposition}
\begin{proof}
	The proof follows from \eqref{eq:bound.z.pointwise} and Lipschitz continuity of $u^\star$.
\end{proof}
\subsection{Proofs for subsection \ref{sec:comparison}}
\begin{proof}[Proof of \Cref{prop:entropic.penal}]
	First observe that for every $\nu\in \cU$ we have
	\begin{equation*}
		\cH(\P^\nu|\P) = \frac12\E^{\P^\smalltext{\nu}}\bigg[\int_0^T\|\nu_s\|^2\d s\bigg],
	\end{equation*}
	so that $\widetilde V^n$ can be re-written as
	\begin{equation}
	\label{eq:control.problem.tilde.nn}
		\widetilde V^n \coloneqq  \sup_{\alpha \in \cA} \E^{\PP^{\smalltext\alpha}}\bigg[G(X^n) - \int_0^T\frac{\varphi(n)}{2}\|\nu_s\|^2+ L\big(s,X^n_{\cdot\wedge s}, u^\star(s, X^n_{\cdot\wedge s},v^{n-1}_s(X^n_{\cdot\wedge s}))\big)\d s \bigg].
	\end{equation}
	It follows by \Cref{prop:Bellman2} that the optimal control problem \eqref{eq:control.problem.tilde.nn} can be characterised by a BSDE in the sense that $\widetilde V^n = \widetilde Y^n_0$ where $(\widetilde Y^n,\widetilde Z^n)$ is the solution of the BSDE
	\begin{equation*}
		\begin{cases}
		\displaystyle \widetilde Y^n_t = F^n(X^n) + \int_t^T\frac{1}{2\varphi(n)}\big\|\widetilde Z^n_s\big\|^2 \d s  - \int_t^T\widetilde Z^n_s\cdot\d B_s,\; t\in[0,T],\; \P\text{\rm--a.s.},\\
		\displaystyle X^n_t =\xi+ \int_0^t\sigma(s,X^n_{\cdot\wedge s})b\big(s,X^n_{\cdot\wedge s},u^\star(s, X^n_{\cdot\wedge s}, v^{n-1}_s(X_{\cdot\wedge s}^n))\big)\d s + \int_0^t\sigma(s,X^n_{\cdot\wedge s})\d B_s,\; t\in[0,T],\; \P\text{\rm--a.s.}
		\end{cases}
	\end{equation*}
	Thus, we have that
	\begin{equation*}
		\widetilde V^n = \varphi(n)\log\bigg( \E^\P\bigg[\exp\bigg(\frac{1}{ \varphi(n)}\bigg(G(X^n) + \int_0^TL(s,X^n_{\cdot\wedge s},  u^\star(s, X^n_{\cdot\wedge s},v^{n-1}_s(X^n_{\cdot\wedge s})))\d s\bigg) \bigg) \bigg] \bigg).
	\end{equation*}
	Since $(\widehat Y^n,Z^n)$ solves BSDE \eqref{eq:BSDE.n}, it follows by Girsanov's theorem that $(X,\widehat Y^n,Z^n)$ satisfies
	\begin{equation*}
		\begin{cases}
		\displaystyle \widehat Y^n_t = F^n(X) + \int_t^T\frac{1}{2\varphi(n)}\big\|Z^n_s\big\|^2 \d s  - \int_t^T Z^n_s\cdot\d B^n_s\; \P^n\text{\rm--a.s.}\\
		\displaystyle X_t =\xi+ \int_0^t\sigma(s,X_{\cdot\wedge s})b\big(s,X_{\cdot\wedge s},u^\star(s, X_{\cdot\wedge s}, Z^{n-1}_s)\big)\d s + \int_0^t\sigma(s,X_{\cdot\wedge s})\d B^n_s,\; t\in[0,T],\; \P^n\text{\rm--a.s.}
		\end{cases}
	\end{equation*}
	with $B^n\coloneqq B -\int_0^tb\big(s,X_{\cdot\wedge s},u^\star(s, X_{\cdot\wedge s}, Z^{n-1}_s)\big)\d s$ and $\P^n$ given in \eqref{eq:def.Pn}.
	Thus, by uniqueness of $(\widetilde Y^n,\widetilde Z^n)$ and of the SDE \eqref{eq:Non.Markov.SDE},
	it follows that $\P^n\circ (X,\widehat Y^n, Z^n)^{-1} = \P\circ (X^n, \widetilde Y^n,\widetilde Z^n)^{-1}$ (recall $Z^{n-1}_t =v_t^{n-1}(X_{\cdot\wedge t})$).
	This implies that
	\begin{equation*}
		\widetilde V^n = \varphi(n)\log\bigg( \E^{\P^n}\bigg[\exp\bigg(\frac{1}{\varphi(n)}\bigg(G(X) + \int_0^TL(s,X_{\cdot\wedge s},  u^\star(s, X_{\cdot\wedge s}, Z^{n-1}_s))\d s\bigg)\bigg) \bigg] \bigg) = V^n.
	\end{equation*}
\end{proof}

\begin{proof}[Proof of \Cref{cor:HJB}]
	We already know that $v^n(s, X_s) = \widehat Y^n_s$ and $(\sigma\nabla v^n)(s, X_s) = Z^n_s$, and by \cite[Theorem 3.1]{ma2002representation} again, we have $v(s,X_s) = Y_s$ and $(\sigma\nabla v)(s,X_s) = Z_s$ for all $s\in [0,T]$.
	By the Markovian property of the driftless SDE \eqref{eq:driftless}, if we start $X$ at time $t$ from $x\in \R^m$, then we have $Y^n_t = v^n(t,x)$ as well as $Y_t = v(t, x)$.
	Moreover, by boundedness of $\sigma$ it follows that $Z$ is bounded.
	Thus, the result follows from \Cref{thm:Rate}. Now, if $\nabla v$ is also Lipschitz continuous, the SDE \eqref{eq:Markov.SDE.strong} has a Lipschitz-continuous drift and  diffusion coefficient. 
	Therefore, it admits a unique strong solution.
	Thus, the optimal control $\nu^\star_t\coloneqq  u^\star(t, X_{ t}^{\nu^\smalltext{\star}}, (\sigma^\top\nabla v)(t, X_{t}^{\nu^\smalltext{\star}}))$ for Problem \eqref{eq:weak.control} gives rise to an optimal control $\alpha^{\star}$ for problem \ref{eq:OG.problem}.
	In fact, 
	we have $\PP\circ (X^{\alpha^\smalltext{\star}})^{-1} =\PP^{\nu^\smalltext{\star}}\circ X^{-1}$, which yields
	\begin{align*}
		V = V_w
 	    &= \E^{\PP^{\smalltext{\nu}^\tinytext{\star}}}\bigg[G(X) - \int_0^TL\big(s, X_{\cdot\wedge s}, \nu^\star_s(X_{\cdot\wedge s})\big)\d s \bigg] = \E^\P\bigg[G(X^{\alpha^\smalltext{\star}}) - \int_0^TL\big(s, X^{\alpha^\smalltext{\star}}_{\cdot\wedge s}, \alpha^\star_t\big)\d s \bigg],
	\end{align*}
	which shows that $\alpha^\star$ is optimal. Similarly, replacing $v$ by $v^n$ in the SDE \eqref{eq:Markov.SDE.strong} allows to obtain a solution $X^n$. 
	In fact, by \cite[Lemma 2.2]{ma2024convergence} (and a direct induction) the function $\nabla v^n$ is also bounded and Lipschitz continuous.
	Standard SDE estimates yield
	\begin{equation*}
		\E^\P\bigg[\sup_{t\in [0,T]}\|X^n_t - X_t^{\alpha^\smalltext{\star}} \|^2\bigg] \le C\E^{\P}\bigg[\int_0^T\|(\sigma\nabla v^n)(t,X^n_t) - (\sigma\nabla v)(t, X_t^{\nu^\smalltext{\star}} )\|^2\d t \bigg].
	\end{equation*} 	
	Let $\P^n$ define the probability measure introduced in \eqref{eq:weak.control}, with $\nu$ therein replaced by $\nu^n\coloneqq  u^\star(t, X_t, (\sigma\nabla v^n)(t, X_t))$.
	Then, we have that $\P^n\circ X^{-1} = \P\circ (X^n)^{-1}$.
	Using Cauchy--Schwarz's inequality as well as boundedness of $b$, we have
	\begin{align*}
		\E^\P\bigg[\sup_{t\in [0,T]}\|X^n_t - X_t^{\alpha^\smalltext{\star}} \|\bigg] &\le \E^{\P^\smalltext{n}}\bigg[\bigg(\frac{\d \P^n}{\d \P} \bigg)^{-1}\int_0^T\|(\sigma\nabla v^n)(t, X_t^n) - (\sigma\nabla v)(t, X^n_t) \|\d t \bigg]\\
			& \quad  +\E^{\P}\bigg[\int_0^T\|(\sigma\nabla v)(t, X_t^n) - (\sigma \nabla v)(t, X_t^{\alpha^\smalltext{\star}} ) \|\d t \bigg]\\
		&\le C \E^{\P}\bigg[\int_0^T\|(\sigma\nabla v^n)(t, X_t) - (\sigma\nabla v)(t, X_t) \|^2\d t \bigg]^{1/2} + C\E^{\P}\bigg[\int_0^T\| X_t^n -  X_t^{\alpha^\smalltext{\star}} \| \d t \bigg] .
	\end{align*}
	By Gr\"onwall's inequality we thus have
	\begin{align*}
		\E^\P\bigg[\sup_{t\in [0,T]}\|X^n_t - X_t^{\alpha^\smalltext{\star}} \| \bigg] &\le C\E^{\P}\bigg[\int_0^T\|Z^n_t - Z_t\|^2\d t\bigg]^{1/2}\le \frac{C}{2^\frac{n}{2}}
	\end{align*}
	where the latter inequality follows by \Cref{prop:lowerbound}. Now, as above, using the triangular inequality, a Girsanov change of measure, as well as boundedness and Lipschitz-continuity of the functions $\sigma$ and $\nabla v$ yields
	\begin{align*}
		\E^{\P}\bigg[\int_0^T\|\nu^{\star,n}_t - \nu^\star_t\|dt\bigg] &\le C\E^{\P}\bigg[\int_0^T\|X^n_t - X_t^{\alpha^\smalltext{\star}} \|\d t\bigg] + C\E^{\P}\bigg[\int_0^T\|Z^n_t - Z_t\|^2\d t\bigg]^{1/2}.
	\end{align*}
\end{proof}

\subsection{Proof of Theorem \ref{thm:Rate.vol}}
\label{sec:controlled.diffusion}

We now turn to the proof of the convergence of the approximation scheme in the controlled volatility case.
Thus in this section, we assume that the function $\sigma$ may further depend on the control process.
In this case, we assume that $\Omega = \cC_m$, the probability measure $\PP$ is the Wiener measure and $X$ is the canonical process.

\medskip
Since for each $\nu \in \Uc$ the {\rm SDE} \eqref{eq:driftless} admits a unique weak solution $\PP^\nu$,
 (after possibly extending the probability space) for every $\nu\in \Uc$, there exists a $\P^\nu$--Brownian motion $B^\nu$ such that
\begin{equation*}
 	X_t  = X_0 + \int_0^t\sigma(s, X_{\cdot\wedge s}, \nu_s)\d B^\nu_s.
 \end{equation*} 
 For each $\nu\in \Uc$, define
 \begin{equation*}
 	\frac{\d\Q^\nu}{\d\PP^\nu} = \cE\bigg(\int_0^Tb(s, X_{\cdot\wedge s}, \nu_s)\d B^\nu_s \bigg).
 \end{equation*}
 Thus, under $\Q^\nu$, the process $X$ satisfies $\d X_t = \sigma(t, X_{\cdot\wedge t}, \nu_t)b(t, X_{\cdot\wedge t}, \nu_t)\d t + \sigma(t, X_{\cdot\wedge t}, \nu_t)\d B^{\Q^\smalltext{\nu}}_t,$ with $B^{\Q^\smalltext{\nu}}_t\coloneqq  B^\nu_t - \int_0^tb(s, X_{\cdot\wedge s}, \nu_s)\d s$.
Consider the control problem
\begin{equation}
\label{eq:vol.control}
	V_w \coloneqq  \sup_{\nu\in \Uc}\E^{\Q^\smalltext{\nu}}\bigg[G(X) - \int_0^TL(s,X_{\cdot\wedge s}, \nu_s)\d s\bigg].
\end{equation}
It follows by \citeauthor*{karoui2013capacities2} \cite[Theorem 4.5]{karoui2013capacities2} that $V = V_w$. The following result plays a key role in our analysis.
It is a version of \Cref{prop:Bellman} for the controlled volatility case.
Therein, we use the unique process
 $\widehat \sigma_s$ such that
\begin{equation*}
	\widehat \sigma_t^2 = \frac{\d[X,X]_t}{\d t},\; \PP\text{\rm--a.s. for all }\PP\in \cP(\cC_m)
\end{equation*}
where $\cP(\cC_m)$ denotes the set of Borel probability measures on $\cC_m$.
The existence of $\widehat\sigma$ is guaranteed by {\rm \citeauthor*{karandikar1995pathwise} \cite{karandikar1995pathwise}}.
\begin{proposition}
\label{prop.Bellman.vol.cont}
	Let the assumptions of {\rm \Cref{thm:Rate.vol}} be satisfied.
	Then, problem \eqref{eq:vol.control} admits an optimal control $\nu^\star$ which satisfies $\nu_t^\star = u^\star(t, X_{\cdot\wedge t}, Z_t,\hat\sigma_t)$ where $(Y,Z)$ solves the {\rm BSDE}
	\begin{equation*}
		Y_t = G(X) - \int_t^TH(s, X_{\cdot\wedge s}, Z_s, \widehat \sigma_s)\d s - \int_t^TZ_s\cdot \d B_s^\star,\;  \PP^\star\text{\rm--a.s.},
	\end{equation*}
	for some probability measure $\PP^\star \coloneqq  \P^{\nu^\ast}$, a $\PP^\star$--Brownian motion $B^\star$ and with
	\begin{equation*}
		u^\star(t,X_{\cdot\wedge t},Z_t,\hat\sigma_t) \in  \underset{\{u\in U: \sigma\sigma^{\smalltext{\top}}(t,X_{\cdot\wedge t},u)=\hat\sigma\}}{\mathrm{argmax}}\big\{ \sigma(t,X_{\cdot\wedge t},u)b(t,X_{\cdot\wedge t},u)\cdot Z_t - L(t, x_{\cdot\wedge t},u)\big\}.
	\end{equation*}
	Moreover, it holds $V_w = Y_0$.
\end{proposition}
\begin{proof}
	Using Girsanov's theorem, it is direct to observe that the optimal control of \eqref{eq:OG.problem} given in the statement gives rise to an optimal control  $\nu^\star$ of \eqref{eq:weak.control}.
	The rest of the proof follows by \cite[Proposition 5.4]{cvitanic2018dynamic}.
\end{proof}
We can now finish the proof.
\begin{proof}[Proof of \Cref{thm:Rate.vol}]
Using the function $u^\star$ and on the Brownian probability space $(\Omega^\star,\cF^\star,\P^\star,B^\star)$ we consider the scheme $(V^n,X^n, Y^n, Z^n,\nu^n)_{n\in\N}$ given by \eqref{eq:problem.n.vol.control} and \eqref{eq:BSDE.n.vol}.
It follows by induction that for each $n$ the {\rm BSDE} \eqref{eq:BSDE.n.vol} admits a unique solution and by Assumption \ref{Ass:Control.Vol}$(ii)$-$(iii)$, the {\rm SDE} \eqref{eq:iterated.SDE.vol} admits a unique solution.
Moreover, applying \Cref{prop:Bellman} as in the proof of \Cref{thm:Rate} (but on the Brownian probability space $(\Omega^\star, \cF^\star,\P^\star,B^\star)$), we have that $V^n= Y^n_0$ satisfies \eqref{eq:Def.Yn.vol}.
Since $V= V_w = Y_0$, in order to conclude the proof it remains to estimate the rate of convergence of $Y^n$ to $Y$.
This would follows exactly as in the proof of \Cref{thm:Rate} if not for the fact that in the current case, the forward process in the iteration scheme depends on $n$ whereas it is fixed in the case of uncontrolled volatility. 
Repeating exactly the computations in the proof of \Cref{prop:lowerbound} and using the Lipschitz-continuity properties in \Cref{Ass:Control.Vol}.$(ii)$ we arrive at
\begin{align*}
(Y_0^n-Y_0)^2+\E^\Q\bigg[\int_0^T\mathrm{e}^{\beta s}\|Z_s^n-Z_s\|^2\mathrm{d}s\bigg] 
&\leq \ell_{G}^2\E^{\Q}\bigg[\sup_{t\in[0,T]}\|X^n_t-X_t\|^2\bigg] + \eta^2\E^\Q\bigg[\int_0^T\mathrm{e}^{\beta s}\|Z_s^{n-1}-Z_s\|^2\mathrm{d}s\bigg]+\frac{C}{\varphi(n)},
\end{align*}
for any $\eta>0$, with the measure $\Q$ introduced in the proof of \Cref{prop:lowerbound} and for some constant $\beta>0$ (possibly depending on $\eta$). Notice that
we have
\begin{equation}
\label{eq:X.by.Z}
	\E^{\P^\smalltext{\star}}\bigg[\sup_{t\in[0,T]}\|X^n_t-X_t\|^2\bigg]\leq 8\ell_{\sigma_u}^2\ell_{u^\smalltext{\star}_\x}^2\mathrm{e}^{ 8(\ell_{\smalltext{\sigma}_\smalltext{\x}}+\ell_{\smalltext{\sigma}_\smalltext{u}}\ell_{\smalltext{u}^\tinytext{\star}_\smalltext{\x}})^\smalltext{2}T}\E^{\P^\smalltext{\star}}\bigg[\int_0^T\|Z^{n-1}_s-Z_s\|^2\mathrm{d}s\bigg].
\end{equation}
In fact,
using Doob's inequality yields
\begin{align*}
\E^{\P^\smalltext{\star}}\bigg[\sup_{t\in[0,T]}\|X^n_t-X_t\|^2\bigg]&\leq 4\E^{\P^\smalltext{\star}}\bigg[\int_0^T\big\|\sigma\big(s,X^n_{\cdot\wedge s}, u^\star(s, X^n_{\cdot\wedge s}, Z^{n-1}_s)\big)-\sigma\big(s,X_{\cdot\wedge s}, u^\star(s, X_{\cdot\wedge s}, Z_s)\big)\big\|^2\mathrm{d}s\bigg]\\
&\leq 4\E^{\P^\smalltext{\star}}\bigg[\int_0^T\bigg(\big(\ell_{\sigma_\x}+\ell_{\sigma_u}\ell_{u^\smalltext{\star}_\x}\big)\sup_{s\in[0,t]}\|X^n_s-X_s\| +\ell_{\sigma_u}\ell_{u^\smalltext{\star}_z}\|Z^n_t-Z_t\|\bigg)^2\mathrm{d}t\bigg]\\
&\leq 8\big(\ell_{\sigma_\x}+\ell_{\sigma_u}\ell_{u^\smalltext{\star}_\x}\big)^2\E^{\P^\smalltext{\star}}\bigg[\int_0^T\sup_{s\in[0,t]}\|X^n_s-X_s\|^2\mathrm{d}t\bigg]+8\ell_{\sigma,u}^2\ell_{u^\smalltext{\star}_\x}^2\E^{\P^\smalltext{\star}}\bigg[\int_0^T\|Z^{n-1}_s-Z_s\|^2\mathrm{d}s\bigg].
\end{align*}
We can then conclude using Gr\"onwall's inequality.
\end{proof}

\begin{appendix}

\section{Bellman optimality principle}
We conclude the paper with this appendix where we recall a well-known result repeatedly used in the paper.
It relates a general non-Markovian stochastic optimal control problem in the weak formulation to solutions of BSDEs.
It is usually referred to in the literature as Bellman optimality principle or martingale optimality principle.
\begin{proposition}
\label{prop:Bellman}
	Let {\rm\Cref{assump:standing}} be satisfied.
	An admissible control $\nu \in \Uc$ is optimal for problem \eqref{eq:weak.control} if and only if it satisfies $\nu_t = u^\star(t, X_{\cdot\wedge t}, Z_t)$, $\d t\otimes\P$--{\rm a.e.}, where $u^\star$ is a Borel-measurable map satisfying
	\begin{equation}
	\label{eq:argmax.pro}
		u^\star(t,\x,z) \in \argmax_{u\in U}\big\{h(t, \x,z,u) \big\},\; (t,\x,z)\in[0,T]\times\Cc_m\times\R^d,
	\end{equation}
	and $(Y,Z)$ is the unique solution of {\rm BSDE}
	\begin{equation}
	\label{eq:BSDE.prop}
		Y_t = G(X) + \int_t^TH(s, X_{\cdot\wedge s}, Z_s)\d s - \int_s^TZ_s\cdot \d B_s, \; t\in [0,T],\; \PP\text{\rm--a.s.},
	\end{equation}
	with
	\begin{equation}\label{eq:bsdeinteggg}
	\E^\P\bigg[\sup_{t\in[0,T]}|Y_t|^2+\int_0^T\|Z_s\|^2\mathrm{d}s\bigg]<+\infty.
	\end{equation}
	Moreover, we have $V_w = Y_0$.
	\end{proposition}
See for instance  \cite[Proposition 2.8]{possamai2021non} for a proof. The next result is similar but for the control problems appearing in the iteration algorithm. 
\begin{proposition}
\label{prop:Bellman2}
	Let {\rm\Cref{assump:standing}} be satisfied. For any $n\in\N$, we have $V^n=Y_0^n$ where the pair $(Y^n,Z^n)$ is the unique solution to the BSDE
\[
	Y^n_t = F^n(X) + \int_t^T\bigg(\frac{1}{2\varphi(n)}\|Z_t^n\|^2+ b\big(s,X_{\cdot\wedge s},u^\star(s, X_{\cdot\wedge s}, Z^{n-1}_s)\big)\cdot Z^n_s\bigg)\mathrm{d}s - \int_t^TZ_s^n\cdot \mathrm{d}B_s, \; t\in [0,T], \; \PP\text{\rm--a.s.},
\]
such that for any $p\geq 1$
\[
\E^\P\bigg[\exp\bigg(p\sup_{t\in[0,T]}|Y_t^n|\bigg)\bigg]+\E^\P\bigg[\bigg(\int_0^T\|Z_s^n\|^2\mathrm{d}s\bigg)^{\frac p2}\bigg]<+\infty.
\]
	\end{proposition}
\begin{proof}
The existence and uniqueness of a solution to the BSDE defined in the statement is immediate from \cite[Corollary 6]{briand2008quadratic}. Now for any $\alpha\in\Ac$, consider the BSDE
\[
Y^{n,\alpha}_t = F^n(X) + \int_t^T\bigg(\frac{\alpha_s}{\varphi^{1/2}(n)}\cdot Z^{n,\alpha}_s-\frac{1}{2}\|\alpha_s\|^2+ b\big(s,X_{\cdot\wedge s},u^\star(s, X_{\cdot\wedge s}, Z^{n-1}_s)\big)\cdot Z^{n,\alpha}_s\bigg)\mathrm{d}s - \int_t^TZ_s^{n,\alpha}\cdot \mathrm{d}B_s, \; t\in [0,T], \; \PP\text{\rm--a.s.}
\]
Notice that this has a unique solution with appropriate integrability, since the BSDE is linear, and we have all the required integrability from the definition of $\Ac$. Besides, it is also clear that
\[
Y^{n,\alpha}_0=\E^{\PP^{\smalltext{\alpha}\smalltext{,}\smalltext{n}}}\bigg[G(X) - \int_0^T\bigg(\frac12\|\alpha_t\|^2 + L\big(t,X_{\cdot\wedge t}, u^\star(t, X_{\cdot\wedge t}, Z^{n-1}_t)\big)\bigg)\mathrm{d}t \bigg]
\]
Now it is immediate that we have
\[
Y^{n,\alpha}_t-Y^n_t\leq \int_t^Tb\big(s,X_{\cdot\wedge s},u^\star(s, X_{\cdot\wedge s}, Z^{n-1}_s)\big)\cdot (Z^{n,\alpha}_s-Z^n_s)\mathrm{d}s-\int_0^T(Z^{n,\alpha}_s-Z^n_s)\cdot\mathrm{d}B_s,\; t\in[0,T],\; \P\text{\rm--a.s.},
\]
which implies that $Y^{n,\alpha}_0\leq Y^n_0$ because we can change measures using Girsanov since $b$ is bounded. In turn, we have $Y^{n}_0\geq V^n$. 

\medskip
Conversely, a standard measurable selection theorem (see for instance \citeauthor*{benes1971existence} \cite{benes1971existence}) ensures that for any $\eps>0$, we can find some $\alpha^\eps\in\Ac$ such that
\[
\frac{1}{2\varphi(n)}\|Z_s^n\|^2\leq \frac{\alpha^{\eps}_s}{\varphi^{1/2}(n)}\cdot Z^{n}_s-\frac{1}{2}\|\alpha_s^\eps\|^2+\eps,\; \d s\otimes\P\text{\rm--a.e.}
\]
We thus conclude as above that $V^n-Y^n_0\geq Y^{n,\alpha^\smalltext{\eps}}_0-Y^n_0\geq \eps T,$ and thus that $V^n=Y^n_0$ by arbitrariness of $\eps$.
\end{proof}
\end{appendix}

{\footnotesize
\bibliography{bibliographyDylan}}

\end{document}